\documentclass[reqno]{amsart} 

\usepackage{amsmath} 
\usepackage{amssymb}  
\usepackage{amsthm}   
\usepackage{amsfonts}
\usepackage{mathtools}
\usepackage{latexsym}
\usepackage{yhmath}    

\usepackage[textwidth=15cm, centering]{geometry} 
\usepackage{setspace}                            
\usepackage{ragged2e}                            
\usepackage{enumitem}                            

\usepackage{graphicx}
\usepackage[dvipsnames]{xcolor}

\usepackage[english]{babel}
\usepackage{csquotes}  

\usepackage[colorlinks=true, citecolor=blue, urlcolor=cyan, linkcolor=blue]{hyperref}

\usepackage{booktabs} 

\numberwithin{equation}{section} 

\newcommand{\parentesis}[1]{\left( #1 \right)}
\newcommand{\corchetes}[1]{\left [ #1 \right ]}
\newcommand{\llaves}[1]{\left \{ #1 \right \}}
 
\newcommand{\abs}[1]{\left \lvert #1 \right \rvert}
\newcommand{\norm}[1]{\left \lVert #1 \right \rVert}
\newcommand{\escalar}[1]{ \big \langle #1 \big \rangle}

\newcommand{\grad}{\nabla}

\newcommand{\fourier}[1]{\widehat{#1}}

\newcommand{\N}{\mathbb{N}}

\newcommand{\RR}{\mathbb{R}}
\newcommand{\RRn}{\mathbb{R}^{n}}

\newtheoremstyle{estiloTeoremas}
  {3pt}                
  {3pt}                
  {\itshape}           
  {}                   
  {\bfseries}          
  {.}                  
  {.5em}               
  {}                   

\newtheoremstyle{estiloNotas}
  {3pt}                
  {3pt}                
  {\upshape}           
  {}                   
  {\bfseries}          
  {.}                  
  {.5em}               
  {}                   

 \theoremstyle{estiloTeoremas}
 \newtheorem{theorem}{Theorem} [section]
 \newtheorem{lemma}[theorem]{Lemma}

 \theoremstyle{estiloNotas}
\newtheorem{definition}[theorem]{Definition}

\newtheorem{remark}{Remark} 

\title[Stability of Stationary solutions to DPME]{Asymptotic stability of Stationary solutions to 3D incompressible flow in porous media with diffusion}

\author[J.S. Ángel-Echeverry]{Juan Sebastián Ángel-Echeverry}
\address[J.S. Ángel-Echeverry]{Instituto de Matem\'atica. Universidade Federal do Rio de Janeiro, CEP 21941-909, Rio de Janeiro - RJ, Brazil}
\email{jsangele@ufrj.br}
\thanks{The author was supported by FAPERJ Nota 10 grant 203.710/2025, and CAPES PROEX grant 88887.901672/2023-00}


\begin{document}

\begin{abstract}
 Under appropiate hypotheses on the external force acting on an incompressible flow diffusing through a porous medium, we show that there is a unique stationary solution to the diffusive porous media equation. Moreover, we show that this solution is asymptotically stable and estimate the decay rate of any perturbation towards this steady state.
\end{abstract}

\maketitle

\section{Introduction}

We are interested in investigating the asymptotic behavior of  stationary solutions to the Diffusive Porous Media Equation,  or DPME,  on $\RR^3$,

\begin{equation} \label{Parabolica}
            \begin{cases} 
         &\partial_{t} \theta+ V_\theta\cdot\nabla \theta = \Delta \theta + f, \\
         &V_\theta= -(\nabla P_{\theta} + \theta \ e_3),\\
         &\nabla \cdot V_\theta = 0,\\
         &\theta(x,0)= \theta_0(x),
            \end{cases}
 \end{equation}
where $\theta$ is flow temperature, $V_{\theta}$ is flow velocity and  $f$ is an external force. This type of equation models phenomena such as groundwater flows and oil recovery processes, see Whitaker \cite{whitaker1986}. 

The first studies regarding flow in porous media were performed by Henry Darcy in the mid 1850s as part of his studies in Hydrology. He discovered that there was a proportional relation between flow rate and the applied pressure gradient, which can be expressed as follows

\begin{equation*}
V_{\theta}= - \kappa(\nabla P_{\theta} + g \theta \, e_3 ),
\end{equation*}
where $\kappa$ is the matrix of medium permeabilities in the different directions divided by the viscosity. Here  $g$ is the acceleration due to gravity and  $e_3 = (0,0,1)$.  In our study, we choose $g= 1, \kappa = Id$. Regarding heat transfer, $\Delta \theta$ represents the net rate of heat conduction, while $ V_{\theta} \cdot \nabla \theta$ represents the net rate of change of temperature due to convection. For further details on the physical behavior of this phenomenon see  Bejan and Nield \cite{nield}

Different aspects of equation \eqref{Parabolica} have been studied. Castro, Córdoba, Gancedo and Orive \cite{Castro_2009} established the existence of strong and weak solutions to the equation with general fractional diffusion $\Lambda^{\alpha} = (- \Delta)^{\frac{\alpha}{2}}$, obtained the decay of solution in $L^{p}$ spaces for $p \geq 2$ and investigated their asymptotic behavior; also they proved the existence of attractors. Meanwhile, Niche and Orive \cite{Niche-Orive} extended the results of Castro, Córdoba, Gancedo and Orive \cite{Castro_2009} regarding the decay of the $L^2$ norm of weak solutions to \eqref{Parabolica}. They also obtained the decay of the $L^p$ norm of the solution and its derivatives for $1 \leq p \leq \infty$, as well as the first order decay of the solutions. Xue \cite{Xue_2009} proved, for $0< \alpha < 1$, global well-posedness for small initial data in a given Besov space. For other results about \eqref{Parabolica} see  Yuan and Yuan \cite{Yuan_Yuan_2009}; Yamazaki \cite{Yamazaki_2011, Yamazaki_2015}; Bie, Wang, and Yao \cite{Wang_Bie_Yao_2016}; Guo and Haifeng \cite{Guo_Haifeng_2019}; Yu, He, Tong and Wu  \cite{Yu_He_Tong_Wu_2013}; and Wu, Yu and Tang \cite{Wu_Yu_Tang_2016}. Regarding the inviscid case, see  Córdoba, Gancedo and Orive \cite{Cordoba_Gancedo_Orive_2007} and Yu and He \cite{Yu_He_2014}.

Our first main result concerns the   existence and uniqueness of stationary solutions to 

\begin{equation} \label{Estacionaria}
     \begin{cases} 
         &V\cdot\nabla \theta_{st} = \Delta  \theta_{st} + f, \\
         &V= -(\nabla P_{ \theta_{st}} +  \theta_{st} \ e_3),\\
         &\nabla \cdot V = 0,
     \end{cases}
\end{equation}
under the hypothesis that the external force $f$ has finite $L^2$ energy  and certain behavior at the small frequencies level. We recall that the  decay character $r^{\ast} = r^{\ast} (g)$ of  $g \in L^{2}(\RR^3)$ is a real number that describes the behavior of  $g$ near the origin in  frequency space, namely $r^{\ast} (g) = r^{\ast}$ provided $|\widehat{g} (\xi)| \sim |\xi|^{r^{\ast}}$ near $\xi = 0$,  for details see Section \ref{decay-character}.

\begin{theorem}  \label{existenciaestacionaria}
Let $f \in L^2(\RR^3)\cap \dot{H}^{-1}(\RR^3)$, with decay character $\gamma^{*} = \gamma^{*} (f) > 2$. Let

\begin{displaymath}
    \norm{f}_X = \max \left\{ \Vert f \Vert _{L^2}, \Vert f \Vert _{\dot{H}^{-1}} \right\}.
\end{displaymath}
Then, there exists a constant $\mu_{\gamma^{*}} > 0$ such that for any $M > \mu_{\gamma^{*}}$, there is a constant $\mathcal{C}_M = \mathcal{C} (M, \gamma^{*}) > 0$  such that if $\norm{f}_X \leq \mathcal{C}_M$, then there is a weak solution  $\theta_{st} \in \dot{H}^{1} (\RR^3)$ to \eqref{Estacionaria}. Moreover, $\lVert \theta_{st} \rVert_{L^2} \leq M$, and $\theta_{st}$ is unique amongst all solutions satisfying

 \begin{equation*}
            \norm{\theta_{st}}_{\dot{H}^1} \leq \norm{f}_{\dot{H}^{-1}}.
        \end{equation*}       
\end{theorem}

In order to solve this problem, we use the approach developed by Bjorland and Schonbek in \cite{bjorlandschonbeck09} to prove  existence of a stationary weak solutions to the  Navier-Stokes equations. This method consists of first constructing a linear parabolic PDE with a structure analogous to \eqref{Parabolica}. Then, integrating the  solution to this linear parabolic PDE in time, we obtain a weak solution to a linear equation with a structure analogous to \eqref{Estacionaria}. We repeat this process iteratively to obtain a sequence  that converges to a weak solution to the nonlinear equation \eqref{Estacionaria}. This method is motivated by the well known  observation that integrating in time the solution to the heat equation yields a solution to the Poisson equation.

 In our second main result, we deal with the asymptotic stability of $\theta_{st}$, in the sense that, for large times, the solution $\theta$ to \eqref{Parabolica} converges to $\theta_{st}$. To that end, we set $\omega = \theta - \theta_{st}$, which satisfies 

\begin{equation} \label{eq_asintotica}
            \begin{cases} 
         &\partial_{t}\omega + V_\theta\cdot\nabla \omega + V_\omega \cdot\nabla \theta_{st}  = \Delta \omega,\\
         &  V_\omega= V_\theta - V_{\theta_{st}}, \\
         &\nabla \cdot V_\omega = 0,\\
         &\omega(x,0)=  \omega_0 = \theta_0 - \theta_{st}.
            \end{cases}
\end{equation}

\begin{theorem} \label{Teorema_Decay_Asintotica}
    Let $\theta$ be a weak solution to \eqref{Parabolica} with initial data $\theta_0 \in L^2(\RR^3)$ and $\theta_{st}$ the stationary solution obtained in Theorem \ref{existenciaestacionaria}. Then, if $\norm{f}_{\dot{H}^{-1}}$ is small enough, the weak solution $\omega$ to \eqref{eq_asintotica} with initial data $\omega_0 = \theta_0 - \theta_{st}$ is such that 

    \begin{equation*} 
       \norm{\omega(t)}_{L^2}^{2} \leq C (1+t)^{- \min \llaves{\frac 32, \frac 32 + \min\llaves{r^{*}(\theta_0), r^{*}(\theta_{st})}}},
    \end{equation*} 
 where $C=C(\norm{\omega_0}_{L^2})$ and $r^{*}(\theta_0)$, $r^{*}(\theta_{st})$ are the decay characters of $\theta_0$ and $\theta_{st}$. 
\end{theorem}
Notice that the initial perturbation $\omega_0$ need not be small. We prove this Theorem using the Fourier Splitting method  developed by Schonbek \cite{ME1980, ME1985}. The method consists of obtaining a differential inequality for the $L^2$ norm of $\omega$, such that, after considering a time-dependent ball in frequency space, its  behavior will  depend  on the small frequencies  of $\fourier{\omega}$ for large t. We prove existence of weak solutions to \eqref{eq_asintotica} in Theorem \ref{teorema_soluciones_debiles_decay}.

The structure of this article is the following. In Section 2, we introduce tools important for the development of the proofs, such as the definition of the Decay Character and the Fourier Splitting method. In Section 3, we focus on proving the existence of weak solutions to \eqref{Estacionaria}, following the method developed by Bjorland and Schonbek \cite{bjorlandschonbeck09}. Finally, in Section 4, we use the retarded mollifiers method developed by Caffarelli, Kohn and Nirenberg \cite{caffarelli1982partial} to prove the existence of weak solutions to \eqref{eq_asintotica}. Then, we address the asymptotic behavior of  \eqref{Parabolica},
applying the Fourier splitting method to prove Theorem \ref{Teorema_Decay_Asintotica}.

\section{Preliminaries}

We adopt the following notations,

\begin{align*}
    & L^{p}(\RR^3) = \llaves{u: \RR^3 \to \RR  \mid \norm{u}_{L^p} < \infty}, \ \text{where} \ \norm{u}_{L^p} = \parentesis{\int_{\RR^3} \abs{ u}^p}^{\frac 1p} \, dx, \, \, \text{for} \ 1 \leq p < \infty, \\ 
     & L^2(\RR^3) \ \text{is the Hilbert space with inner product} \ \escalar{u,v}_{L^2} = \int_{\RR^3} uv  \ dx,\\
    &\dot{H}^{1}(\RR^3) \ \text{is the closure of the space $C^{\infty}_{c}(\RR^3)$ in the norm $\norm{\nabla u}_{L^2}$},\\
    &\dot{H}^{-1}(\RR^3) = (\dot{H}^{1}(\RR^3))^{'},\\
    & \escalar{g, u}_{\dot{H}^{-1}, \dot{H}^{1}} \ \text{is the action of the functional $g \in H^{-1}(\RR^3)$ over the element $u \in \dot{H}^{1}(\RR^3)$},\\
    & \norm{f}_X= \max \llaves{\norm{f}_{L^2}, \norm{f}_{\dot{H}^{-1}}} \ \forall f\in L^{2}(\RR^3)\cap\dot{H}^{-1}(\RR^3), 
\end{align*}
where $V^{'}$ denotes the dual space of $V$. Also, $C$ denotes general positive constants that may change from line to line.

\subsection{Decay Character} \label{decay-character} Given a function  $u_0 \in L^2(\RR^n)$ the decay character associates to $u_0$ a number $\gamma = \gamma(u_0)$  which measures the algebraic order of $u_{0}$ near the origin, in the sense that $\gamma = \gamma(u_0)$ if $|\widehat{u_0} (\xi)| \approx |\xi| ^{\gamma}$, for $|\xi| \approx 0$. This concept was first introduce Bjorland and Schonbek in \cite{bjorland2009}; it was later studied in more detail by Niche and Schonbek \cite{CESARELENA2015}, and Brandolese \cite{brandolese2016}. We now introduce some important definitions and properties regarding the decay character.

\begin{definition} \label{decayind}
 The decay indicator of $u_0 \in L^2(\RRn)$ is $\gamma \in \parentesis{- \frac n2, \infty}$ such that

\begin{equation} \label{indicator}
    P_\gamma(u_0) = \lim_{\rho \to 0^{+}}\rho^{-2\gamma}\rho^{-n}\int_{B(\rho)}\abs{\fourier{u_0}(\rho)}^2 d\rho, 
\end{equation}
provided the limit exists. Here,  $B (\rho)$ is the ball of radius $\rho$ centered at the origin. 
\end{definition}

\begin{definition} \label{decaych}
The decay character of $u_0 \in L^2(\RRn)$, denoted by $\gamma ^{\ast} = \gamma^{*}(u_0)$, is the only number $\gamma \in \parentesis{-\frac n2, \infty}$ such that $0 < P_{\gamma}(u_0) < \infty$.
\end{definition}

\begin{remark}
It is possible to explicitly calculate the decay character when $u_0 \in L^p(\RRn) \cap L^{2}(\RRn)$ with $1\leq p \leq 2$, see Ferreira, Niche and Planas \cite{ferreiranicheplanas2017}.  
\end{remark}

A useful application concerning estimates for the decay of solutions to the heat equation through the use of the decay character concept was established by Bjorland and Schonbek \cite{bjorland2009}. 

\begin{theorem}
\label{theorem-decay-heat}
Let $u_0 \in L^2(\RR^n)$ with decay character $\gamma^{*}=\gamma^{*}(u_{0})$ and $u$ be a solution to the heat equation with initial condition $u_0$. Then, there exist constants $C_1, C_2 >0$ such that

\begin{equation*}
    C_1(1+t)^{-\parentesis{\frac n2+ \gamma^{*}}} \leq \norm{u(t)}_{2}^{2}\leq C_2(1+t)^{-\parentesis{\frac n2+ \gamma^{*}}}.
\end{equation*}
\end{theorem}

The following Lemma is an immediate consequence of Theorem \ref{theorem-decay-heat}.

\begin{lemma} \label{lema0.1} Let $f \in L^2(\RRn)$ such that $\gamma^{*} (f) > 1 - \frac{n}{2}$ and  let $\Phi = e^{\Delta t}f$. Then 
    
\begin{equation} \label{lema2.1.1}
        \int_0^{\infty} \norm{\Phi (t)}_{L^2}^2  dt \leq \mu_{\gamma^{*}} \norm{f}_{L^2}^2 \leq  \mu_{\gamma^{*}} (1 + \norm{f}_{L^2}^2) = \mathcal{M}_1, 
\end{equation}
where $\mu_{\gamma^{*}} > 0$ is a constant that depends only on $\gamma^{*}$ and $n$.
\end{lemma}

Niche and Schonbek \cite{CESARELENA2015} generalized Theorem \ref{theorem-decay-heat} as follows. Let $X$ be a Hilbert space and $\mathcal{L}$ be a linear  diagonalizable  pseudodifferential operator $\mathcal{L}:X^n \to (L^2(\RRn))^n$ with  symbol $\mathcal{M}(\xi)$ such that

\begin{equation*}
    \mathcal{M}(\xi) = P^{-1}(\xi)D(\xi)P(\xi), \hspace{0.5 cm} \xi-\text{a.e.},
\end{equation*}
where $P(\xi) \in O(n)$ and $D(\xi)=-c_i\abs{\xi}^{2\alpha}\delta_{ij}$ with $c_i>C>0$ and $0< \alpha \leq 1$. Then, the following holds.

\begin{theorem} \label{th2.5}
Let $u_0 \in L^2(\RR^n)$ with decay character $\gamma^{*}=\gamma^{*}(u_0)$ and $u$ a solution to

\begin{equation} \label{lin}
    u_t= \mathcal{L}u,
\end{equation}
with initial condition $u_0$. Then there exist constants $C_1, C_2 >0$ such that

\begin{equation*}
    C_1(1+t)^{-\parentesis{\frac{1}{\alpha}}\parentesis{\frac n2+ \gamma^{*}}} \leq \norm{u(t)}_{2}^{2}\leq C_2(1+t)^{-\parentesis{\frac{1}{\alpha}}\parentesis{\frac n2+ \gamma^{*}}}.
\end{equation*}
\end{theorem}

 In Definitions \ref{decayind} and \ref{decaych}, we assumed that the limit \eqref{indicator} exists. However, Brandolese \cite{brandolese2016} proved  that there exist functions in $L^2(\RR^n)$ such that \eqref{indicator} is not well-defined. Such functions can be constructed by introducing fast oscillations  around the origin in the frequency space, see Example 6.1 in Brandolese \cite{brandolese2016}. To circumvent this problem, he gave generalized definitions of upper and lower decay character for functions such that \eqref{indicator} does not exist. These definitions are characterized in terms of  subsets of homogeneous Besov spaces ${\dot B}_{2, \infty}^{-\sigma}$, see Brandolese \cite{brandolese2016}. 

We now prove a simple, but important, Lemma concerning the decay character of $f \pm g$, where $f, g \in L^2(\RR^n)$.

\begin{lemma} \label{suma_decay_character}
    Let $f, g \in L^2(\RR^n)$, with decay characters $r^{\ast} (f), r^{\ast} (g)$ respectively. Then

    \begin{displaymath}
        r^{\ast} (f \pm g) \leq \min \left\{ r^{\ast} (f), r^{\ast} (g)\right\}.
    \end{displaymath}
\end{lemma}

\begin{proof}[Proof.] We have 

\begin{displaymath}
    \rho^{-2\gamma}\rho^{-n}\int_{B(\rho)}\abs{\fourier{(f \pm g)}(\rho)}^2 d\rho    \leq C \rho^{-2 \gamma}\rho^{-n}\int_{B(\rho)}\abs{\fourier{f}(\rho)}^2 d \rho + C \rho^{-2 \gamma}\rho^{-n}\int_{B(\rho)}\abs{\fourier{g}(\rho)}^2 d \rho.
\end{displaymath}
If $\gamma = \min \left\{ r^{\ast} (f), r^{\ast} (g)\right\}$, the result follows after taking limit when $\rho$ goes to zero and using the definition of decay character. 
\end{proof}

\section{Existence of solution to the Stationary DPM equation \eqref{Estacionaria}}

In this section we will prove the existence of a unique weak solution to \eqref{Estacionaria} under some constraints related to the size of the external force $f$. To address this problem, we first construct the following  linear parabolic PDE, whose structure is  analogous to \eqref{Parabolica}

\begin{equation} \label{auxiliar1}
 \begin{split}
        u_{t}^{i+1}(t) + V^i \cdot \nabla  u^{i+1}(t) &= \Delta  u^{i+1}(t) - V^i \cdot \nabla \Phi (t),  \\
        V^i & = -(\nabla P + \theta_{st}^{i} \ e_3),\\
        u^{i+1}(0)&=0,
\end{split}    
\end{equation}
for given stationary scalar $\theta_{st}^{i}$ and vector field $V^i$. In Theorem \ref{th2}, we prove existence of weak solutions to \eqref{auxiliar1}.

We also consider the following PDE  

\begin{equation} \label{auxiliar2}
            \begin{cases} 
         &V^{i}\cdot\nabla \theta_{st}^{i+1} = \Delta \theta_{st}^{i+1} + f, \\
         &V^{i}= -(\nabla P + \theta_{st}^{i} \ e_3),\\
         &\nabla \cdot V^{i} = 0,
            \end{cases}
\end{equation}
 where again $\theta_{st}^{i}$ is a known stationary scalar. In Theorem \ref{th1}, we prove the existence of a weak solution to \eqref{auxiliar2}. Then, to the solution to \eqref{auxiliar1} we add a corrector term given by the solution to the heat equation with initial datum $f$, which integrated in time provide another solution to \eqref{auxiliar2}. Uniqueness of solution given by Theorem \ref{th1} implies that the solution to \eqref{auxiliar2} is  actually  this integral.
Therefore, for each $i \in \N$ we obtain a solution $\theta_{st}^{i+1}$ to \eqref{auxiliar2} and with this we construct a sequence ${\theta_{st}^{i+1}}$ that converges to a solution to \eqref{Estacionaria}. This method is motivated by the well known  observation that integrating in time the solution to the heat equation yields a solution to the Poisson equation.

A useful result regarding the $L^p$ norm of the term $V_{\theta}$, proved by Niche and Orive \cite{Niche-Orive}, states that 

    \begin{equation} \label{remark2}
        \norm{V_{\theta}(t)}_{L^p} \leq C  \norm{\theta(t)}_{L^p}, \hspace{0,5 cm} 1< p < \infty.
    \end{equation}
Analogously,

\begin{equation*} 
        \norm{V^{i}}_{L^p} \leq C  \norm{\theta^{i}_{st}(t)}_{L^p}, \hspace{0,5 cm} 1< p < \infty.
    \end{equation*}

\begin{theorem} \label{th1}
Let $f\in \dot{H}^{-1}(\RR^3)$ and $\theta_{st}^{i}\in \dot{H}^{1}(\RR^3) \cap L^2(\RR^3)$ be a fixed scalar field. Then there exists a unique weak solution $\theta_{st}^{i+1}\in \dot{H}^{1}(\RR^3)$ to \eqref{auxiliar2}, in the sense that for any $ \varphi \in \dot{H}^1(\RR^3)$

\begin{equation*}
    \escalar{ V^{i}\cdot\nabla{\theta_{st}^{i+1}}, \varphi} = - \escalar{ \nabla  \theta_{st}^{i+1}, \nabla \varphi} + \escalar{f, \varphi}.
\end{equation*}
Moreover, this solution satisfies

\begin{equation} \label{E3.1}
   \norm{\nabla \theta_{st}^{i+1}}_{L^2} \leq \norm{f}_X. 
\end{equation}
\end{theorem}

\begin{proof}
We use the Galerkin approximation method to prove the existence of weak solutions in smooth bounded domains $\Omega \subset \RR^3$ with Dirichlet boundary conditions. In order to prove the existence of weak solutions  in the whole space we construct a sequence of solutions defined on  balls with increasing radii and show its convergence.  

Let $\mathcal{W}= \llaves{w_k}_{k \in \N}$ be the set of eigenfunction associated to the Laplace Operator in $L^2(\Omega)$ with Dirichlet boundary conditions. We recall it  is an orthonormal basis for $L^2 (\Omega)$; an orthogonal basis  for  $\dot{H}^{1}_0(\Omega)$ and $w_k \in C^{\infty}(\Omega)$ for all $k \in \N$. For every $N \in \N$ we define

\begin{equation*}
    W_N = span \llaves{w_k}_{k=1,\dots, N},
\end{equation*}
and consider the projection of  \eqref{auxiliar2} onto the subspace $W_N$. Let us define, for  fixed  $N \in \N$,

\begin{equation*}
   (\theta_{st})^{i+1}_{N}(x)= \sum_{k=1}^{N} c_k w_k(x),
\end{equation*}
then for each $s=1,\dots, N$ define

\begin{equation} \label{3.1.2}
    \escalar{ V^{i}\cdot\nabla{(\theta_{st})^{i+1}_{N}, w_s }}  + \escalar{ \nabla  \parentesis{\theta_{st}}_N^{i+1}, \nabla w_s} = \escalar{f, w_s}.
\end{equation}
This is a system of ordinary differential equations in the finite-dimensional space $W_N$  which has a unique solution $(\theta_{st})^{i+1}_{N}$. Since $\nabla  \cdot V^{i}=0$ it follows that sequence $\llaves{(\theta_{st})^{i+1}_{N}}_{N \in \N}$ satisfies the following  inequality

\begin{equation} \label{energy_N}
\norm{\nabla (\theta_{st})^{i+1}_{N}}_{L^2} \leq C\norm{f}_{X}.
\end{equation}
Then $\llaves{(\theta_{st})^{i+1}_{N}}_{N \in \N}$ is uniformly bounded in $\dot{H}_{0}^{1}(\Omega)$, and  we obtain a subsequence $\llaves{(\theta_{st})^{i+1}_{N}}_{N \in \N}$, that we continue denoting in the same way\footnote{From now on we omit saying that we rename the subsequence.}, such that

\begin{equation} \label{conv_debil}
(\theta_{st})^{i+1}_{N} \to \theta_{st}^{i+1} \hspace{0.5 cm} \text{strongly in $L^2(\Omega)$}, \hspace{1 cm} (\theta_{st})^{i+1}_{N} \to \theta_{st}^{i+1} \hspace{0.5 cm} \text{weakly in $\dot{H}^1_{0}(\Omega)$}.
\end{equation}
Notice that $(\theta_{st})^{i+1}_{N}$ satisfies

\begin{equation*}
   \escalar{ V^{i}\cdot\nabla{(\theta_{st})_N^{i+1}}, \varphi_m }  + \escalar{ \nabla  (\theta_{st})_N^{i+1}, \nabla \varphi_m} = \escalar{f, \varphi_m}
\end{equation*} 
for all $\varphi_m \in  W_m$. Then, due to \eqref{conv_debil}, we have that

\begin{equation} \label{form_debil}
   \escalar{ V^{i}\cdot\nabla \theta_{st}^{i+1}, \varphi_m }  + \escalar{ \nabla  \theta_{st}^{i+1}, \nabla \varphi_m} = \escalar{f, \varphi_m}
\end{equation}
for all $\varphi_m \in  W_m$.

Since $\mathcal{W}$ is an orthogonal basis for $\dot{H}^{1}_{0}(\Omega)$, each $\varphi \in \dot{H}^{1}_{0}(\Omega)$ can be written as

\begin{equation*}
\varphi (x) = \lim_{m \to \infty} \sum_{s=1}^{m} b_s w_s
\end{equation*}
Using this strong convergence in \eqref{form_debil} we obtain

\begin{equation*} 
   \escalar{ V^{i}\cdot\nabla \theta_{st}^{i+1}, \varphi }  + \escalar{ \nabla  \theta_{st}^{i+1}, \nabla \varphi} = \escalar{f, \varphi}
\end{equation*}
for all $\varphi \in \dot{H}^{1}_{0}(\Omega)$.

We now extend the result to the whole space as follows. For $n \in \N$ define $B_n= B(0,R_n)$, where $\llaves{R_n}$ is an increasing sequence of positive real numbers such that $R_n \to \infty$, when $n \to \infty$.

Since $B_n$ is a smooth bounded domain of $\RR^3$  there is a unique solution $(\hat{{\theta}}_{st})^{i+1}_n$ of the Dirichlet problem over $B_n$. We can extend this function to all $\RR^3$ as follows

\begin{equation*}
   (\theta_{st})^{i+1}_{n}(x)= \begin{cases}
        (\hat{\theta}_{st})^{i+1}_{n}(x), \hspace{1 cm} &\text{in} \hspace{0,1 cm} B_n,\\
        0, \hspace{1 cm} &\text{elsewhere}.
    \end{cases}
\end{equation*}
Since we have Dirichlet boundary conditions, it follows that

\begin{equation*}
\begin{split}
 \norm{ (\theta_{st})^{i+1}_{n}}_{\dot{H}^{1}(\RR^3)}^2 = \norm{ (\hat{\theta}_{st})^{i+1}_{n}}_{\dot{H}^{1}(B_n)}^2,
 \end{split}   
\end{equation*}
then $ (\theta_{st})^{i+1}_{n} \in \dot{H}^{1}(\RR^3)$, for all $n \in \N$. As a result

\begin{equation} \label{1.2.3}
    \norm{\nabla   (\theta_{st})^{i+1}_{n}}_{L^2(\RR^3)} = \norm{ (\nabla \hat{\theta}_{st})^{i+1}_{n}}_{L^2(B_n)}^2
  \leq \norm{f}_X, \hspace{0,5 cm} n\in \N.
\end{equation}
So $\llaves{ (\theta_{st})^{i+1}_{n}}_{n \in \N}$ is a uniformly bounded sequence in $\dot{H}^1(\RR^3)$ and there is a subsequence such that

\begin{equation} \label{T1E2}
    (\theta_{st})^{i+1}_{n} \rightharpoonup (\theta_{st})^{i+1} \hspace{0,1 cm} \text{weakly in } \hspace{0,1 cm} \dot{H}^1(\RR^3).  
\end{equation}
Now, for each $M \in \N$ $\llaves{ (\theta_{st})^{i+1}_{n} |_{B_M}}_{n \in \N}$ is uniformly bounded. Then, there is a subsequence such that 

\begin{equation*}
    (\theta_{st})^{i+1}_{n} \to \theta_{st}^{i+1} \hspace{0,1 cm} \text{strongly in } \hspace{0,1 cm} L^2(B_M).  
\end{equation*}
So, for $M=1$ there is a subsequence $\llaves{ (\theta_{st})_{n,1}^{i+1} |_{B_1}}$ of $\llaves{ (\theta_{st})^{i+1}_{n} |_{B_1}}$  such that 

\begin{equation*}
    (\theta_{st})^{i+1}_{n,1} \to (\theta_{st})_{1}^{i+1} \hspace{0,1 cm} \text{strongly in } \hspace{0,1 cm} L^2(B_1).  
\end{equation*}
Using a diagonal argument we have that there is a subsequence $\llaves{ (\theta_{st})_{n}^{i+1}}_{n \in \N}$ that converges  strongly in $L^2(B_M)$, for all $M\in \N$. From \eqref{T1E2} it also follows, taking a subsequence if necessary, that  $\llaves{ (\theta_{st})_{n}^{i+1}}_{n \in \N}$ converges weakly in $\dot{H}^1(\RR^3)$ to $(\theta_{st})^{i+1}$.

Now, let $\varphi_{n_0} \in C^{\infty}_c (\RR^3)$ such that $\text{supp} (\varphi_{n_0}) \subset B_{n_0}$ for some sufficiently large $n_0 \in \N$. For $m > n_0$ we have that $B_{n_0} \subset B_m$, so $\varphi_{n_0} \in C^{\infty}_c (B_m)$. Then, since $(\theta_{st})^{i+1}_{n}$ is a weak solution to the equation in $B_m$, we get

\begin{equation*}
    \escalar{ V^{i}\cdot\nabla{(\theta_{st})^{i+1}_{n}}, \varphi_{n_0} }  + \escalar{ \nabla  (\theta_{st})^{i+1}_{n}, \nabla \varphi_{n_0}} = \escalar{f, \varphi_{n_0}}.
\end{equation*}
Now, since $\llaves{(\theta_{st})^{i+1}_{n}}_{n \in \N}$  converges weakly to $(\theta_{st})^{i+1}$ in $\dot{H}^1(\RR^3)$ and $\varphi_{n_0} \in C^{\infty}_c (\RR^3)$, we obtain  

\begin{equation*}
    \escalar{ \nabla  (\theta_{st})^{i+1}_{n}, \nabla \varphi_{n_0}} \to \escalar{ \nabla  (\theta_{st})^{i+1}, \nabla \varphi_{n_0}}, \hspace{0,2 cm} \text{when} \hspace{0,2 cm} n \to \infty.
\end{equation*}
On the other hand, we notice that 

\begin{align*}
     \abs{\escalar{ V^{i}\cdot\nabla\parentesis{(\theta_{st})^{i+1}_{n}-(\theta_{st})^{i+1}}, \varphi_{n_0} }} &= \abs{\escalar{ V^{i}\cdot\nabla \varphi_{n_0}, (\theta_{st})^{i+1}_{n}-(\theta_{st})^{i+1}}}\\
     &\leq C\norm{V^i}_{L^2(B_m)} \norm{ (\theta_{st})^{i+1}_{n}-(\theta_{st})^{i+1}}_{L^2(B_m)}\\ 
     &\leq C\norm{(\theta_{st})^i}_{L^2(B_m)} \norm{ (\theta_{st})^{i+1}_{n}-(\theta_{st})^{i+1}}_{L^2(B_m)},  
\end{align*}
where $C= \sup_{\text{supp}(\varphi_{n_0})} \abs{\nabla \varphi_{n_0}}$. Since $\llaves{(\theta_{st})^{i+1}_{n}}_{n \in \N}$  converges strongly to $(\theta_{st})^{i+1}$ in $L^2(B_m)$ for all $m \in \N$, it follows that 

\begin{equation*}
    \lim_{n \to \infty} \escalar{ V^{i}\cdot\nabla\parentesis{(\theta_{st})^{i+1}_{n}-(\theta_{st})^{i+1}}, \varphi_{n_0} }_{L^2} = 0 \hspace{0,5 cm} \text{in $B_m$},
\end{equation*}
for all $m \in \N$. So we proved that for any $\varphi_{n_0} \in C^{\infty}_{c}(\RR^3)$
\begin{equation*}
\begin{split}
  \escalar{ V^{i}\cdot\nabla{(\theta_{st})^{i+1}}, \varphi_{n_0}}  + \escalar{ \nabla  (\theta_{st})^{i+1}, \nabla \varphi_{n_0}}&= \lim_{n \to \infty} \parentesis{\escalar{ V^{i}\cdot\nabla{(\theta_{st})^{i+1}_{n}}, \varphi_{n_0} }  + \escalar{ \nabla  (\theta_{st})^{i+1}_{n}, \nabla \varphi_{n_0}}}\\
  &= \escalar{f, \varphi_{n_0}}.
\end{split} 
\end{equation*}
Now, since $\dot{H}^{1}(\RR^3) = \overline{C^{\infty}_c(\RR^3)}$, for any $\varphi \in \dot{H}^{1}(\RR^3)$ we have that

\begin{equation} \label{3.1.5}
  \escalar{ V^{i}\cdot\nabla{(\theta_{st})^{i+1}}, \varphi}  + \escalar{ \nabla  (\theta_{st})^{i+1}, \nabla \varphi} = \escalar{f, \varphi}.
\end{equation}
Then $(\theta_{st})^{i+1}$ is a weak solution.

Now, we prove uniqueness of this solution. Let us suppose that there are two weak solutions $(\theta_{st})_1$ and $(\theta_{st})_2$  and define $\tilde{\theta}_{st}= (\theta_{st})_1 -(\theta_{st})_2$. Since $\tilde{\theta}_{st} \in \dot{H}^{1}(\RR^3)$, from \eqref{3.1.5} we have

\begin{equation*}
    \escalar{ V^{i}\cdot\nabla{\tilde{\theta}_{st}}, \tilde{\theta}_{st}}  + \escalar{ \nabla  \tilde{\theta}_{st}, \nabla \tilde{\theta}_{st}} = 0,
\end{equation*}
and, since $ \escalar{ V^{i}\cdot\nabla{\tilde{\theta}_{st}}, \tilde{\theta}_{st}}= 0$, it follows

\begin{equation*}
     \norm{\tilde{\theta}_{st}}_{\dot{H}^1(\RR^3)}=\norm{\nabla \tilde{\theta}_{st}}_{L^2}^2=\escalar{ \nabla  \tilde{\theta}_{st}, \nabla \tilde{\theta}_{st}} = 0.
\end{equation*}
This implies $\tilde{\theta}_{st}=0$ and $(\theta_{st})_1 = (\theta_{st})_2$.

To prove the estimate \eqref{E3.1} we take $\liminf$ in \eqref{1.2.3} and by  the semi-continuity of the norm we get

\begin{equation*}
    \norm{\nabla (\theta_{st})^{i+1}}_{L^2} \leq \norm{f}_X.
\end{equation*}
\end{proof}

\begin{theorem} \label{th2}

Let $\theta_{st}^{i} \in \dot{H}^{1}(\RR^3)\cap L^2(\RR^3)$, satisfy 

\begin{equation*}
  \norm{\nabla \theta_{st}^{i}}_{L^2} \leq \norm{f}_X,
\end{equation*}
where $f \in L^2(\RR^3)$, $\gamma^{*}= \gamma^{*}(f) > 2$, and $ \Phi (t) = e^{\Delta t}f$. There exist a unique weak solution $u^{i+1} \in L^{\infty}(\RR^+,L^2(\RR^{3}))\cap L^{2}(\RR^+,\dot{H}^1(\RR^{3}))$ to \eqref{auxiliar1}, in the sense that  for any  $\varphi \in \dot{H}^{1}_0(\RR^3)$ and a.e $t \in \RR^{+}$,

\begin{equation*}
    \escalar{ u_{t}^{i+1}(t), \varphi}_{L^2} + \escalar{ V^i \cdot \nabla  u^{i+1}(t), \varphi}_{L^2}= - \escalar{ \nabla  u^{i+1}(t), \nabla \varphi}_{L^2} - \escalar{ V^i \cdot \nabla \Phi (t), \varphi}_{L^2}.
\end{equation*}
Also, this solution verifies 

\begin{equation} \label{EA2.1}
    \sup_t \norm{u^{i+1}(t)}_{L^2}^2+ \int_{0}^{\infty} \norm{ \nabla u^{i+1}(s)}_{L^2}^2 ds \leq C \norm{f}_X^3 \corchetes{\mu_{\gamma^{*}} (1+\norm{f}_X^2)}^\frac12,
\end{equation}
where $\mu_{\gamma^{*}}>0$ is a constant.
   
\end{theorem}

\begin{proof}

The proof is based on the Galerkin approximation method and is similar to the proof of Theorem \ref{th1} with some modifications due to the parabolic nature of the equation, for details we refer the reader to Chapter 9, Salsa and Verzini \cite{salsa_libro}.

\end{proof}

Having proved the existence of weak solutions to  \eqref{auxiliar1} and \eqref{auxiliar2}, we now will describe and prove some decay properties of  $u^{i+1}$  that will be useful when we study the relation between them.

\subsection{Decay properties}

Our goal now is to prove that the decay of the solution $u^{i+1}$ to \eqref{auxiliar1} is given by

\begin{equation*} 
       \norm{u^{i+1}(t)}_{L^2}^{2} \leq \overline{C} \parentesis{\norm{\theta_{st}^{i}}_{L^{2}}^2+1}^5\parentesis{1+\norm{f}_{X}^2}^2\norm{f}_{X}^{3} (1+t)^{-\frac52},
\end{equation*}
where $\overline{C}=\overline{C}(\gamma^{*}(f))$.  To this end, we use the Fourier Splitting Method. We first prove some Lemmas.

\begin{lemma}  \label{lema1}
 Let $u^{i+1}$ be the solution to equation \eqref{auxiliar1} given by Theorem \ref{th2}. Then,

\begin{equation*} 
    \abs{\fourier{u^{i+1}}(\xi,t)} \leq C\abs{\xi} \norm{\theta_{st}^{i}}_{L^{2}} \parentesis{ \int_{0}^{t} \norm{u^{i+1}(s)}_{L^{2}} ds +\mu_{\gamma^{*}}(1+\norm{f}_{X})}.
\end{equation*}    
\end{lemma} 

\begin{proof}

Taking Fourier Transform in \eqref{auxiliar1} we get

\begin{equation*}
   \frac{d}{dt} \fourier{u^{i+1}}(\xi,t) + \xi\fourier{V^{i}(\xi)u^{i+1}(\xi,t)} = \abs{\xi}^2 \fourier{u^{i+1}}(\xi,t) - \xi \fourier{V^{i}(\xi)\Phi(\xi,t)}.
\end{equation*}
Solving this ordinary differential equation, and taking into account that $u^{i+1}(0)=0$, we obtain

\begin{equation*}
\begin{split}
    \fourier{u^{i+1}}(\xi,t)&= e^{-\abs{\xi}^2 t} \fourier{u^{i+1}}(\xi,0) - \int_{0}^{t}e^{-\abs{\xi}^2 (t-s)} (\xi\fourier{V^{i}(\xi)u^{i+1}(\xi,s)}+\xi \fourier{V^{i}(\xi)\Phi(\xi,s)}) ds\\
    &=- \int_{0}^{t}e^{-\abs{\xi}^2 (t-s)} (\xi\fourier{V^{i}(\xi)u^{i+1}(\xi,s)}+\xi \fourier{V^{i}(\xi)\Phi(\xi,s)}) \ ds.\\
\end{split}    
\end{equation*}
Then, 

\begin{align*}
    \abs{ \fourier{u^{i+1}}(\xi,t)}  \leq \abs{\xi} \int_{0}^{t} \abs{ \fourier{V^{i}(\xi)u^{i+1}(\xi,s)}+ \fourier{V^{i}(\xi)\Phi(\xi,s)} } ds  \leq \abs{\xi} \int_{0}^{t} \parentesis{\abs{ \fourier{V^{i}(\xi)u^{i+1}(\xi,s)}}+ \abs{\fourier{V^{i}(\xi)\Phi(\xi,s)} }} ds. 
\end{align*}
Estimating each term of the right hand side in the last inequality using that $\norm{\fourier{\varphi}}_{L^{\infty}} \leq \norm{\varphi}_{L^{1}}$, together with the fact that $\norm{V^{i}}_{L^{2}} \leq C \norm{\theta_{st}^{i}}_{L^{2}}$ (see \eqref{remark2}) and the Hölder Inequality we get

\begin{equation*}
\begin{split}
        \abs{ \fourier{V^{i}(\xi)u^{i+1}(\xi,s)}} &\leq \norm{\fourier{V^{i}(\xi)u^{i+1}(\xi,s)}}_{L^{\infty}} \leq \norm{V^{i}u^{i+1}(s)}_{L^{1}}\\
        &\leq \norm{V^{i}}_{L^{2}}\norm{u^{i+1}(s)}_{L^{2}}
        \leq C \norm{\theta_{st}^{i}}_{L^{2}}\norm{u^{i+1}(s)}_{L^{2}},
\end{split}
\end{equation*}
and 

\begin{equation*}
    \abs{ \fourier{V^{i}(\xi)\Phi(\xi,s)}} \leq \norm{\fourier{V^{i}(\xi)\Phi(\xi,s)}}_{L^{\infty}} \leq \norm{V^{i}\Phi (s)}_{L^{1}} \leq \norm{V^{i}}_{L^{2}}\norm{\Phi (s)}_{L^{2}} \leq C\norm{\theta_{st}^{i}}_{L^{2}}\norm{\Phi (s)}_{L^{2}}.
\end{equation*}
Now, from Lemma \ref{lema0.1} and $\gamma^* > 2$,  we have that

\begin{equation*}
     \int_0^{\infty} \norm{\Phi (t)}_{L^2}  dt \leq \mu_{\gamma^{*}}(1+\norm{f}_{L^2}).
\end{equation*}
Then 

\begin{align*}
    \abs{\fourier{u^{i+1}}(\xi,t)}  \leq C \abs{\xi} \norm{\theta_{st}^{i}}_{L^{2}} \parentesis{ \int_{0}^{t} \norm{u^{i+1}(s)}_{L^{2}} ds +\mu_{\gamma^{*}}(1+\norm{f}_{X})} \ ds.
\end{align*}
\end{proof}

Now, using Lemma \ref{lema1}, we establish  a differential inequality for $\norm{u^{i+1}(t)}_{L^2}^{2}$. 

\begin{lemma} \label{decay_u_lema_2} 
    Let $u^{i+1}$ be the solution to equation \eqref{auxiliar1}  given by Theorem \ref{th2}. Then, for $m\geq 3$ we have that

    \begin{equation} \label{estimativa2}
    \begin{split}
        \frac{d}{dt}  \parentesis{\norm{u^{i+1}(t)}_{L^2}^{2} (1+t)^m }   
      &\leq C \norm{f}_{X}^2 \norm{\Phi (t)}_{L^3}^2 (1+t)^m\\ & +  C  \norm{\theta_{st}^{i}}_{L^{2}}^2 \parentesis{ \int_{0}^{t} \norm{u^{i+1}(s)}_{L^{2}} ds +\mu_{\gamma^{*}}(1+\norm{f}_{X})}^2 (1+t)^{m-\frac72}.
    \end{split}        
    \end{equation}
\end{lemma}

\begin{proof}
     We start by multiplying equation \eqref{auxiliar1} by $u^{i+1}$. Then, we obtain

\begin{equation*}
     \frac12 \frac{d}{dt} \norm{u^{i+1}(t)}_{L^2}^{2} + \norm{ \nabla u^{i+1}(t)}_{L^2}^{2}  = \escalar{V^{i} \cdot \nabla u^{i+1}(t), \Phi (t)}.
\end{equation*}
By Hölder inequality,  Sobolev Embedding Theorem  and  our hypothesis on $\theta_{st}^i$, we have that

\begin{equation*}
   \frac12 \frac{d}{dt} \norm{u^{i+1}(t)}_{L^2}^{2} +  \norm{ \nabla u^{i+1}(t)}_{L^2}^{2} \leq \frac{C}2 \norm{f}_{X}^2 \norm{\Phi (t)}_{L^3}^2  + \frac12\norm{ \nabla u^{i+1}(t)}_{L^2}^2.
\end{equation*}
Rewriting this last inequality we get

 \begin{equation} \label{ec1}
   \frac{d}{dt} \norm{u^{i+1}(t)}_{L^2}^{2} +  \norm{ \nabla u^{i+1}(t)}_{L^2}^{2} \leq C \norm{f}_{X}^2 \norm{\Phi (t)}_{L^3}^2 . 
 \end{equation}
 At this point we are going to use the Fourier Splitting Method. We start as usual by using Plancherel Theorem and splitting the space into a ball $B(R)$ and its complement, where $R=R(t)$ will be a time dependent radius. Namely,

  \begin{align*}
     -\norm{ \nabla u^{i+1}(t)}_{L^2}^{2}  & =  -\norm{ \fourier{\nabla u^{i+1}(t)}}_{L^2}^{2}      =  -\norm{  \abs{\xi} \fourier{u^{i+1}}(\xi,t)}_{L^2}^{2}\\
      &= -\int_{B(R)} \abs{\xi}^2 \abs{\fourier{u^{i+1}}(\xi,t)}^2 \hspace{0,1 cm} d\xi - \int_{B(R) ^c} \abs{\xi}^2 \abs{\fourier{u^{i+1}}(\xi,t)}^2 \hspace{0,1 cm} d\xi\\
      & \leq  - R(t)^2 \int_{B(R) ^c} \abs{\fourier{u^{i+1}}(\xi,t)}^2 \hspace{0,1 cm} d\xi.
 \end{align*}
 Then, 

\begin{equation*}
    -\norm{ \nabla u^{i+1}(t)}_{L^2}^{2}  + R(t)^2 \norm{\fourier{u^{i+1}(t)}}_{L^2}^2    
    \leq   R(t)^2 \int_{B(R)}  \abs{\fourier{u^{i+1}}(\xi,t)}^2 \hspace{0,1 cm} d\xi. 
\end{equation*}
Using this last inequality in \eqref{ec1} we obtain

\begin{equation*}
     \frac{d}{dt} \norm{u^{i+1}(t)}_{L^2}^{2} + R(t)^2 \norm{u^{i+1}(t)}_{L^2}^2   \leq C \norm{f}_{X}^2 \norm{\Phi (t)}_{L^3}^2  + R(t)^2 \int_{B(R)}  \abs{\fourier{u^{i+1}}(\xi,t)}^2 \hspace{0,1 cm} d\xi . 
\end{equation*}
Now, using Lemma \ref{lema1} in the second term of the right hand side we have that

\begin{align*}
      \frac{d}{dt} &\norm{u^{i+1}(t)}_{L^2}^{2} + R(t)^2 \norm{u^{i+1}(t)}_{L^2}^2\\ 
      & \leq C \norm{f}_{X}^2 \norm{\Phi (t)}_{L^3}^2 + C \norm{\theta_{st}^{i}}_{L^{2}}^2 R(t)^2 \int_{B(R)}  \abs{\xi}^2  \parentesis{ \int_{0}^{t} \norm{u^{i+1}(s)}_{L^{2}} ds +\mu_{\gamma^{*}}(1+\norm{f}_{X})}^2 \hspace{0,1 cm} d\xi\\      
       & \leq  C \norm{f}_{X}^2 \norm{\Phi (t)}_{L^3}^2 +   C R(t)^7  \norm{\theta_{st}^{i}}_{L^{2}}^2 \parentesis{ \int_{0}^{t} \norm{u^{i+1}(s)}_{L^{2}} ds +\mu_{\gamma^{*}}(1+\norm{f}_{X})}^2.
\end{align*}
Finally, setting $R(t)^2 = \dfrac{m}{(1+t)}$ and using  $h(t) = (1+t)^m$ as an integrating factor it follows that

\begin{align*}
    \frac{d}{dt}  \parentesis{\norm{u^{i+1}(t)}_{L^2}^{2} (1+t)^m }  
      & \leq C \norm{f}_{X}^2 \norm{\Phi (t)}_{L^3}^2 (1+t)^m\\
      &+  C  \norm{\theta_{st}^{i}}_{L^{2}}^2 \parentesis{ \int_{0}^{t} \norm{u^{i+1}(s)}_{L^{2}} ds +\mu_{\gamma^{*}}(1+\norm{f}_{X})}^2 (1+t)^{m-\frac72},
\end{align*}
where $C = C(m) > 0$.
\end{proof}

In the following Theorem we will use Lemma \ref{decay_u_lema_2} to estimate the decay of $u^{i+1}$.

\begin{theorem}
    Let $u^{i+1}$ be the solution to \eqref{auxiliar1} given by Theorem \ref{th2}. Then $u^{i+1}$ satisfies the following bound

   \begin{equation} \label{ec3.1}
       \norm{u^{i+1}(t)}_{L^2}^{2} \leq \overline{C}\parentesis{\norm{\theta_{st}^{i}}_{L^{2}}^2+1}^5\parentesis{1+\norm{f}_{X}^2}^2\norm{f}_{X}^{3} (1+t)^{-\frac52},
\end{equation}
where $\overline{C}=\overline{C}(m, \gamma^{*})>0$.
\end{theorem}

\begin{proof}
    Let us recall that by \eqref{EA2.1}, we have  
\begin{align*}
    \norm{u^{i+1}(t)}_{L^2}^{2} &  \leq C \norm{f}_X^3 \corchetes{\mu_{\gamma^{*}} (1+\norm{f}_X^2)}^\frac12 \leq C \norm{f}_X^3 (1+\norm{f}_X^2),
\end{align*}
so

\begin{align*}
     \parentesis{ \int_{0}^{t} \norm{u^{i+1}(s)}_{L^{2}} ds +\mu_{\gamma^{*}}(1+\norm{f}_{X})}^2
     \leq C \norm{f}_X^{3} (1+\norm{f}_{X}^2) (1+t)^2.
 \end{align*}
 Using this last inequality in \eqref{estimativa2} we get

\begin{equation} \label{decay_u_eq_1}
\begin{split}
     \frac{d}{dt}  &\parentesis{\norm{u^{i+1}(t)}_{L^2}^{2} (1+t)^m }
      \leq C \norm{f}_{X}^2 \norm{\Phi (t)}_{L^3}^2 (1+t)^m\\  & \hspace{3.5 cm}+  C  \norm{\theta_{st}^{i}}_{L^{2}}^2 \parentesis{ \int_{0}^{t} \norm{u^{i+1}(s)}_{L^{2}} ds +\mu_{\gamma^{*}}(1+\norm{f}_{X})}^2 (1+t)^{m-\frac72}\\ 
      & \leq C \norm{f}_{X}^2 \norm{\Phi (t)}_{L^3}^2 (1+t)^m  +   \overline{C} \norm{\theta_{st}^{i}}_{L^{2}}^2\norm{f}_{X}^{3}(1+\norm{f}_{X}^2)(1+t)^{m-\frac32}\\
      & \leq C \norm{f}_{X}^2 \norm{\Phi (t)}_{L^3}^2 (1+t)^m  +   \overline{C} \parentesis{\norm{\theta_{st}^{i}}_{L^{2}}^2 +1 }\norm{f}_{X}^{3}(1+\norm{f}_{X}^2)  (1+t)^{m-\frac32}.
\end{split}     
\end{equation}
We now integrate \eqref{decay_u_eq_1} in time. First, using the decay of the $L^2$ norm of the heat equation and its gradient we get

\begin{equation*}
    \norm{\Phi (t)}_{L^2}  \norm{\nabla \Phi (t)}_{L^2} \leq (1+t)^{- \parentesis{\frac32 + \gamma^*+\frac12}} = (1+t)^{- \parentesis{2 + \gamma^*}}
\end{equation*}
Then, we can use this estimate in combination with Lebesgue Interpolation  and Sobolev Embedding Theorem in the first term on the right hand side of \eqref{decay_u_eq_1} to get

\begin{align*}
      \int_{0}^{t}C \norm{f}_{X}^2 \norm{\Phi (s)}_{L^3}^2 (1+s)^m ds & =  C \norm{f}_{X}^2 \int_{0}^{t}  \norm{\Phi (s)}_{L^3}^2 (1+s)^m ds\\     
      & \leq  C \norm{f}_{X}^2 \int_{0}^{t}  \norm{\Phi (s)}_{L^2}\norm{ \nabla \Phi (s)}_{L^2} (1+s)^m ds \\
      & \leq  C \norm{f}_{X}^2 \int_{0}^{t} (1+s)^{ m- \parentesis{2 + \gamma^*}} ds \leq  \ \overline{C}\norm{f}_{X}^2 (1+t)^{m-\parentesis{1+\gamma^*}}.
\end{align*}
Now, integrating directly the second term in the right hand side of \eqref{decay_u_eq_1} we get
\begin{align*}
   \int_0^t  \overline{C} \parentesis{\norm{\theta_{st}^{i}}_{L^{2}}^2 +1 }&\norm{f}_{X}^{3}(1+\norm{f}_{X}^2)  (1+s)^{m-\frac32} \hspace{0,1 cm} ds\\   
   &\leq  \overline{C} \parentesis{\norm{\theta_{st}^{i}}_{L^{2}}^2 +1 }\norm{f}_{X}^{3}(1+\norm{f}_{X}^2)(1+t)^{m-\frac12}.
\end{align*}
Then,

\begin{align*}
    \norm{u^{i+1}(t)}_{L^2}^2 &\leq  \overline{C} \norm{f}_{X}^2 (1+t)^{-\parentesis{1+\gamma^*}}
      +  \overline{C}  \parentesis{\norm{\theta_{st}^{i}}_{L^{2}}^2 +1 }\norm{f}_{X}^{3}(1+\norm{f}_{X}^2)(1+t)^{-\frac12}\\
    &\leq   \overline{C} \parentesis{\norm{\theta_{st}^{i}}_{L^{2}}^2 +1 } \parentesis{1 +\norm{f}_X^2}\norm{f}_{X}^3(1+t)^{- \min\llaves{1+\gamma^*,\frac12}}\\
    & =  \overline{C} \parentesis{\norm{\theta_{st}^{i}}_{L^{2}}^2 +1 } \parentesis{1 +\norm{f}_X^2}\norm{f}_{X}^3(1+t)^{- \frac12}.   
\end{align*}
From this

\begin{align*}
     \int_{0}^{t} \norm{u^{i+1}(s)}_{L^{2}} ds 
     \leq \corchetes{ \overline{C} \parentesis{\norm{\theta_{st}^{i}}_{L^{2}}^2 +1 } \parentesis{1 +\norm{f}_X^2}\norm{f}_{X}^3}^\frac12 (1+t)^{\frac34},
 \end{align*}
 then

 \begin{align*}
     &\frac{d}{dt}  \parentesis{\norm{u^{i+1}(t)}_{L^2}^{2} (1+t)^m }\\  
      &\leq C \norm{f}_{X}^2 \norm{\Phi (t)}_{L^3}^2 (1+t)^m  +   \overline{C} \norm{\theta_{st}^{i}}_{L^{2}}^2 \parentesis{ \int_{0}^{t} \norm{u^{i+1}(s)}_{L^{2}} ds +\mu_{\gamma^{*}}(1+\norm{f}_{X})}^2 (1+t)^{m-\frac72}\\     
      & \leq C \norm{f}_{X}^2 \norm{\Phi (t)}_{L^3}^2 (1+t)^m  +  \overline{C} \parentesis{\norm{\theta_{st}^{i}}_{L^{2}}^2+1}^2\parentesis{1+\norm{f}_{X}^2}^2\norm{f}_{X}^{3} (1+t)^{m-2}.\\      
\end{align*}
As $\gamma^{*} > 2$
    
\begin{align*}
  \norm{u^{i+1}(t)}_{L^2}^{2} &\leq  \overline{C}\norm{f}_{X}^2 (1+\theta_{st})^{-\parentesis{1+\gamma^*}} +\overline{C}\parentesis{\norm{\theta_{st}^{i}}_{L^{2}}^2+1}^2\parentesis{1+\norm{f}_{X}^2}^2\norm{f}_{X}^{3} (1+t)^{-1}\\ 
   &\leq \overline{C}\parentesis{\norm{\theta_{st}^{i}}_{L^{2}}^2+1}^2\parentesis{1+\norm{f}_{X}^2}^2\norm{f}_{X}^{3}(1+t)^{- \min{\llaves{1+\gamma^*,1}}}.
\end{align*}
This same process, repeated three times, produces the optimal decay rate
\begin{equation*}
     \norm{u^{i+1}(t)}_{L^2}^{2} \leq  \overline{C} \parentesis{\norm{\theta_{st}^{i}}_{L^{2}}^2+1}^5\parentesis{1+\norm{f}_{X}^2}^2\norm{f}_{X}^{3} (1+t)^{-\frac52}.
\end{equation*}
\end{proof}

\subsection{The relation between the solutions to the auxiliary equations  \eqref{auxiliar1} and \eqref{auxiliar2}}

Having proved the decay properties of the $L^2$ norm of the solution $u^{i+1}$ to \eqref{auxiliar1}, we are now in position to study the relation between this solution and the solution $\theta_{st}^{i+1}$ to \eqref{auxiliar2}. The main goal is prove that
\begin{equation*}
    \theta_{st}^{i+1}(x)= \int_0^{\infty} u^{i+1}(x,t) + \Phi (x,t) \hspace{0,1 cm} dt,
\end{equation*}
where $\Phi (t)= e^{\Delta t} f$. The path to accomplishing this end consists of the following. First, we define $z^{i+1}(x,t) = u^{i+1}(x,t) + \Phi (x,t)$ and prove that its integral in time is the limit of a Cauchy sequence in $L^{2}(\mathbb{R}^3)$. Next, we prove that this integral is, in fact, a weak solution to \eqref{auxiliar2}, thus, by the uniqueness arising from Theorem \ref{th1}, we arrive at the relation we are looking for.

\begin{lemma} \label{lema_3.6}
    Let $u^{i+1}$ be the solution to \eqref{auxiliar1} given by Theorem \ref{th2} and $\Phi (t)= e^{\Delta t}f$. Then, $z^{i+1}(t)= u^{i+1}(t) + \Phi (t)$ is such that 

    \begin{equation*}
        \int_{0}^{\infty} z^{i+1}(t) \hspace{0,1 cm} dt \in L^2(\RR^3)
    \end{equation*}
\end{lemma}

\begin{proof}
    Fix $i$ and define

\begin{equation*}
    Z^{i+1}_n= \int_{0}^{n} z^{i+1}(t) \hspace{0,1 cm} dt.
\end{equation*}
Let us prove that $Z^{i+1}_n \in L^{2}(\RR^3)$. We have that

\begin{equation*}
\begin{split}
    \norm{Z^{i+1}_n}_{L^2}  \leq \int_{0}^{n} \norm{z^{i+1}(t)}_{L^2} \hspace{0,1 cm} dt &\leq \int_{0}^{n} \norm{u^{i+1}(t)}_{L^2} \hspace{0,1 cm} dt + \int_{0}^{n} \norm{\Phi (t)}_{L^2} \hspace{0,1 cm} dt\\
    & \leq \int_{0}^{n} \norm{u^{i+1}(t)}_{L^2} \hspace{0,1 cm} dt + \mu_{\gamma^{*}}(1+\norm{f}_{L^2}^2)\\
    & \leq \overline{C}^\frac12 \parentesis{\norm{\theta_{st}^{i}}_{L^{2}}^2+1}^\frac52\parentesis{1+\norm{f}_{X}^2}\norm{f}_{X}^{\frac32} + \mu_{\gamma^{*}}(1+\norm{f}_{L^2}^2),
\end{split}    
\end{equation*}
where in the second line we used Lemma \ref{lema0.1} and in the third line we used \eqref{ec3.1}. So $Z^{i+1}_n \in L^2(\RR^3)$ for all $n \in \N$ and the sequence $\llaves{Z^{i+1}_n}_{n \in \N}$ is uniformly bounded.

Now, we will prove that the sequence is Cauchy. For all $n \in \N$ define

\begin{equation*}
    a_n = \int_0^n \norm{z^{i+1}(t)}_{L^2} \hspace{0,1 cm} dt.
\end{equation*}
As before we have that

\begin{align*}
    a_n &\leq  \int_{0}^{n} \norm{u^{i+1}(t)}_{L^2} \hspace{0,1 cm} dt + \int_{0}^{n} \norm{\Phi (t)}_{L^2} \hspace{0,1 cm} dt\\
    & \leq \overline{C}^\frac12 \parentesis{\norm{\theta_{st}^{i}}_{L^{2}}^2+1}^\frac52\parentesis{1+\norm{f}_{X}^2}\norm{f}_{X}^{\frac32} (1+n)^{-\frac14} + \frac{1}{\frac14 -\frac{\gamma^{*}}{2}} (1+n)^{-\frac14 -\frac{\gamma^{*}}{2}}\\
    &\leq \left( \overline{C}^\frac12 \parentesis{\norm{\theta_{st}^{i}}_{L^{2}}^2+1}^\frac52\parentesis{1+\norm{f}_{X}^2}\norm{f}_{X}^{\frac32} \right) (1+n)^{-\frac14}.
\end{align*}
Therefore $a_n \to 0$ as $n \to \infty$, so is a Cauchy sequence. Notice that, for all $\varepsilon >0$

\begin{equation*}
    \norm{Z^{i+1}_m - Z^{i+1}_n}_{L^2} \leq \int_0^m \norm{z^{i+1}(t)}_{L^2} \hspace{0,1 cm} dt - \int_0^n \norm{z^{i+1}(t)}_{L^2} \hspace{0,1 cm} dt \leq \abs{a_m - a_n} < \varepsilon, 
\end{equation*}
if $m > n > n_0$. Then $\llaves{Z^{i+1}_n}_{n \in \N}$ is a convergent sequence in $L^2(\RR^3)$.
\end{proof}

\begin{lemma}
     Let $u^{i+1}$ be the solution to \eqref{auxiliar1} given by Theorem \ref{th2}. Then the function $z^{i+1}(t)= u^{i+1}(t) + \Phi (t)$ satisfies

     \begin{equation*}
         \theta_{st}^{i+1}(x)= \int_0^{\infty} z^{i+1}(x,t) \hspace{0,1 cm} dt.
     \end{equation*}    
\end{lemma}

\begin{proof}
    Let

    \begin{equation*}
        Z= \int_0^{\infty} z^{i+1}(x,t) \hspace{0,1 cm} dt.
    \end{equation*}
In order to prove that $Z$ is a weak solution to \eqref{auxiliar2}, let $\llaves{Z^{i+1}_n}_{n\in \N}$ be as in the previous Lemma. As $u^{i+1}$ is a weak solution to \eqref{auxiliar1} and  $\Phi (t)$ is a classical solution to the Heat Equation, we have that

  \begin{equation*}
     \escalar{ z_t^{i+1}, \varphi} + \escalar{ V^i \cdot \nabla  z^{i+1}, \varphi}= - \escalar{ \nabla  z^{i+1}, \nabla \varphi}, 
\end{equation*} 
for all $\varphi \in \dot{H}^{1}(\RR^3)$ and almost every $t \in [0, \tau]$ for $\tau >0$. Integrating over $(0,n)$ in time we get

\begin{equation} \label{3.2.1}
    \escalar{ u^{i+1}(n) +\Phi (n), \varphi}   + \escalar{V^i  \cdot \nabla Z^{i+1}_n, \varphi}_{L^2}=  -  \escalar{ \nabla  Z_n^{i+1}, \nabla \varphi}_{L^2} +\escalar{f, \varphi}. 
\end{equation}
As $\Phi (t) \to 0$ as $t \to \infty$ and  by \eqref{ec3.1}, the  first term on the left hand side in \eqref{3.2.1} goes to $0$ as $n \to \infty$.

Notice that for  $\varphi \in C^{\infty}_c (\RR^3)$, we have that $\Delta \varphi \in C^{\infty}_c (\RR^3) \subset L^2(\RR^3)$ and  $\nabla \varphi \in L^3(\RR^3)$. Then,  by Hölder inequality and  the Sobolev Embedding Theorem,  we get

\begin{align*}
    \escalar{V^i  \cdot \nabla (Z^{i+1}_n -Z)} 
     \leq C \norm{ \nabla V^i}_{L^2} \norm{Z^{i+1}_n -Z}_{L^2} \norm{ \nabla \varphi}_{L^3} \to 0,
\end{align*}
due to  strong convergence of $\llaves{Z^{i+1}_n}_{n\in \N}$ as in Lemma \ref{lema_3.6} and $\nabla V^{i} \in L^2(\RR^3)$.

Analogously, 

\begin{equation*}
    -  \escalar{ \nabla (Z^{i+1}_n -Z), \nabla \varphi}_{L^2} \to 0 \hspace{0,5 cm} \text{as} \hspace{0,5 cm} n \to \infty.
\end{equation*}
So far we have that

\begin{equation*}
    \escalar{V^i  \cdot \nabla Z, \varphi}_{L^2}=  -  \escalar{ \nabla  Z, \nabla \varphi}_{L^2} +\escalar{f, \varphi} 
\end{equation*}
for all $\varphi \in C^{\infty}_c (\RR^3)$. Using the fact that $\overline{ C^{\infty}_c (\RR^3)} = \dot{H}^{1}(\RR^3)$ we get 

\begin{equation*}
    \escalar{V^i  \cdot \nabla Z, \varphi}_{L^2}=  -  \escalar{ \nabla  Z, \nabla \varphi}_{L^2} +\escalar{f, \varphi} 
\end{equation*}
for all $\varphi \in \dot{H}^1(\RR^3)$. Then $Z$ is a weak solution to \eqref{auxiliar2}, and by  uniqueness from Theorem \ref{th2}, we obtain the desired result.  \end{proof}

 Finally, we prove that $\theta_{st}^{i+1}$ is bounded by a term that depends on the $L^2$ norm of $\theta_{st}^i$. This will be helpful in the next Section.
\begin{lemma}
     Let $\theta_{st}^{i+1}$ be the solution to \eqref{auxiliar2} given by Theorem \ref{th1}. Then the function $\theta_{st}^{i+1}$ satisfies

     \begin{equation} \label{3.3.0}
     \norm{\theta_{st}^{i+1}}_{L^2}^2 \leq \overline{C}\parentesis{\norm{\theta_{st}^{i}}_{L^{2}}^2+1}^5\parentesis{1+\norm{f}_{X}^2}^2\norm{f}_{X}^3 +  \mathcal{M}_1^2
 \end{equation} 
 where $ \mathcal{M}_1= \mu_{\gamma^{*}}(1+\norm{f}_{L^2}^2).$
\end{lemma}

\begin{proof} [Proof.]
    Let $z^{i+1}(t) = u^{i+1}(t) + \Phi (t)$ where $\Phi (t) = e^{\Delta t}f$ and $u^{i+1}$ is the weak solution to \eqref{auxiliar1} given by Theorem \ref{th2}. We have that
 \begin{equation*}
     \norm{\theta_{st}^{i+1}}_{L^2} \leq \int_{0}^{\infty} \norm{u^{i+1}(t)}_{L^2}  \hspace{0,1 cm} dt + \int_{0}^{\infty} \norm{\Phi (t)}_{L^2}  \hspace{0,1 cm} dt,
 \end{equation*}
 and from \eqref{ec3.1} and  \eqref{lema2.1.1} it follows that

 \begin{equation*}
     \norm{\theta_{st}^{i+1}}_{L^2} \leq \overline{C}(m, \gamma^*)^{\frac12} \parentesis{\norm{\theta_{st}^{i}}_{L^{2}}^2+1}^\frac52\parentesis{1+\norm{f}_{X}^2}\norm{f}_{X}^{\frac32} +  \mathcal{M}_1^2.
 \end{equation*}
 \end{proof}

\subsection{Weak Solution to the stationary equation}

At this point, we have developed all the tools needed to prove the existence of weak solutions to equation \eqref{Estacionaria}. To this end, we first prove a Lemma establishing that, under certain conditions on $\norm{f}_X$, the $L^2$ norm of $\theta_{st}^i$ is bounded by a constant $M$. Following this, we prove the main Theorem of this section regarding the existence  of weak solutions, and the uniqueness under specific conditions.

\begin{lemma} \label{lema4.1}
     Let $\theta_{st}^{i+1}$ be the solution to \eqref{auxiliar2} given by Theorem \ref{th1}. Also, let $M$ a positive constant such that $M >\mu_{\gamma^{*}}$. Additionally suppose that $f \in L^2(\RR^3)$ and $\gamma^{*}(f)>2$. Then there is a constant $\mathcal{K}_M$ such that if $\norm{\theta_{st}^{i}}_{L^2} \leq M$  and 

\begin{equation*}
   \norm{f}_X \leq \max \llaves{\mathcal{K}_M, \parentesis{\frac{M- \mu_{\gamma^{*}}}{\mu_{\gamma^{*}}}}},
\end{equation*}
then

\begin{equation*}
    \norm{\theta_{st}^{i+1}}_{L^2} \leq M.
\end{equation*}
\end{lemma} 

\begin{proof} 

 By hypothesis we have that

 \begin{equation*}
     \parentesis{\norm{\theta_{st}^{i}}_{L^{2}}^2+1}^5 \leq \parentesis{M^2+1}^5. 
 \end{equation*}
Then, from \eqref{3.3.0} we obtain

 \begin{equation} \label{4.1.0}
     \norm{\theta_{st}^{i+1}}_{L^2}^2 \leq \overline{C}\parentesis{M^2+1}^5\parentesis{1+\norm{f}_{X}^2}^2\norm{f}_{X}^3 +  \mathcal{M}_1^2.
 \end{equation} 
 Let $Y= \norm{f}_X$ and define the polynomial in $Y$,

 \begin{equation*}
    \parentesis{1+Y^2}^2 Y^3  = \frac{M^2 - \mathcal{M}_1^2}{ \overline{C}\parentesis{M^2+1}^5}. 
 \end{equation*}
 If this polynomial has a positive root $\mathcal{K}_M$ we will have that, when $\norm{f}_X \leq \mathcal{K}_M$,
 
 \begin{equation*}
     \norm{\theta_{st}^{i+1}}_{L^2}^2 < M^2.
 \end{equation*}
 In order to prove this, notice that the polynomial has a positive root when the term in the right hand side is positive. As $\overline{C}\parentesis{M^2+1}^5 \geq 0$ we only need that $M^2 - \mathcal{M}_1^2 \geq 0$. In fact,  by hypothesis

 \begin{equation*}
     \norm{f}_X^2 < \frac{M-\mu_{\gamma^{*}}}{\mu_{\gamma^{*}}}
 \end{equation*}
 then 

 \begin{equation*}
      \mathcal{M}_1 = \mu_{\gamma^{*}}(1+\norm{f}_X^2) < M.
 \end{equation*}
So $M^2 - \mathcal{M}_1^2 \geq 0$, and we have proved that there exists a root $\mathcal{K}_M>0$. 
\end{proof}

We are now able to prove Theorem \ref{existenciaestacionaria}. 
\begin{proof} [Proof of Theorem \ref{existenciaestacionaria}] Let $M$ and $f$ as in Lemma \ref{lema4.1}. Our goal is to find a unique $\theta_{st} \in \dot{H}^{1} (\RR^3)$ such that 

   \begin{equation} \label{existencia_estacionaria_eq1}
      \escalar{V\cdot \grad \theta_{st}, \varphi}_{L^2} = \escalar{\nabla \theta_{st}, \nabla \varphi}_{L^2} + \escalar{f,\varphi}_{L^2}, \hspace{0,5  cm} \forall \varphi \in \dot{H}^{1}(\RR^3).
   \end{equation}  
{\bf Existence:}
Our first objective is to find a convergent sequence  $\llaves{\theta_{st}^{i}}_{i\in \N}$ in $\dot{H}^1(\RR^3)$ such that its limit satisfies \eqref{existencia_estacionaria_eq1}. Let us start by choosing $\theta_{st}^0 \in \dot{H}^1(\RR^3)$ such that 

\begin{equation*}
    \norm{\theta_{st}^0}_{L^2} \leq M, \hspace{1 cm}     \norm{\nabla \theta_{st}^0}_{L^2}^2 \leq \norm{f}_{\dot{H}^{-1}}^2.
\end{equation*}
By Theorem \ref{th1} there exists a weak solution $\theta_{st}^1$ to 

\begin{equation*}
   \begin{cases} 
   &V^{0}\cdot\nabla{\theta_{st}^{1}} = \Delta \theta_{st}^{1} + f \\
    &V^{0}= (\nabla P + \theta_{st}^{0} \ e_3)\\
    &\nabla \cdot V^{0} = 0,
    \end{cases}
\end{equation*}
such that $\theta_{st}^1 \in \dot{H}^1(\RR^3)$ and $\norm{\nabla \theta_{st}^1}_{L^2}^2\leq \norm{f}_{\dot{H}^{-1}}^2$, such that by Lemma \ref{lema4.1}  $\norm{\theta_{st}^1}_{L^2} \leq M$.
Repeating this process recursively we obtain a sequence $\llaves{\theta_{st}^{i}}_{i \in \N}$ of functions in $\dot{H}^1(\RR^3)$ such that $\theta_{st}^{i+1}$ is a weak solution to \eqref{auxiliar2} and

\begin{equation*}
    \norm{\theta_{st}^{i+1}}_{L^2} \leq M, \hspace{1 cm}     \norm{\nabla \theta_{st}^{i+1}}_{L^2}^2 \leq \norm{f}_{\dot{H}^{-1}}^2, \hspace{1 cm} \forall i \in \N.
\end{equation*}
Now, defining $Z^{i+1}= \theta_{st}^{i+1} - \theta_{st}^{i}$, we notice that it satisfies 

\begin{equation} \label{4.2.0}
    V^{i} \cdot\nabla Z^{i+1} + (V^{i}- V^{i-1})\cdot \nabla \theta_{st}^{i} = \Delta Z^{i+1}.
\end{equation}
Multiplying \eqref{4.2.0} by $Z^{i+1}$,  integrating by parts over $\RR^3$ and using the skew-symmetry property of the convective term  we have that

\begin{align*}
    \norm{\nabla Z^{i+1}}_{L^2}^2  
    \leq C\norm{ \nabla Z^{i}}_{L^2} \norm{\nabla Z^{i+1}}_{L^2} \norm{\theta_{st}^{i}}_{L^3}
    \leq  \frac 12 C\norm{ \nabla Z^{i}}_{L^2}^2 \norm{\theta_{st}^{i}}_{L^3}^2 + \frac12 \norm{\nabla Z^{i+1}}_{L^2}^2,
\end{align*}
where we used the 2-3-6 Hölder Inequality, Sobolev Embedding Theorem  and Young inequality for products. Now, by Lebesgue Interpolation, Lemma \ref{lema4.1} and $\norm{ \nabla \theta_{st}^{i}}_{L^2} < \norm{f}_X$, we have

\begin{equation} \label{existencia_estacionaria_eq2}
    \norm{\nabla Z^{i+1}}_{L^2}^2 \leq C M\norm{f}_X \norm{ \nabla Z^{i}}_{L^2}^2 \hspace{1 cm} \forall i \in \N.
\end{equation}
For $Z^{i}$, we also have that

\begin{equation} \label{existencia_estacionaria_eq3}
    \norm{\nabla Z^{i}}_{L^2}^2 \leq C M\norm{f}_X \norm{ \nabla Z^{i-1}}_{L^2}^2. 
\end{equation}
Plugging \eqref{existencia_estacionaria_eq3} into \eqref{existencia_estacionaria_eq2} and using a recursive argument we obtain

\begin{equation} \label{existencia_estacionaria_eq4}
    \norm{\nabla Z^{i+1}}_{L^2}^2 \leq (C M\norm{f}_X)^{i+1} \norm{ \nabla Z^{1}}_{L^2}^2 \leq  2\norm{f}_X  (C M\norm{f}_X)^{i+1}.
\end{equation}
Now, we choose the constant $\mathcal{C}_M$ such that

\begin{equation*}
   \mathcal{C}_M= \min \llaves{\frac{1}{C M}, \mathcal{K}_M}
\end{equation*}
where $\mathcal{K}_M$ is given by the Lemma \ref{lema4.1}. By hypothesis we have that

\begin{equation*}
    \norm{f}_X \leq \mathcal{C}_M \leq \frac{1}{C M},
\end{equation*}
then

\begin{equation*}
    C M\norm{f}_X < 1.
\end{equation*}
Taking  $i \to \infty$ in equation \eqref{existencia_estacionaria_eq4}, we obtain

\begin{equation*}
   \norm{\nabla( \theta_{st}^{i+1} - \theta_{st}^{i})}_{L^2}^2  =\norm{\nabla Z^{i+1}}_{L^2}^2  \to 0.
\end{equation*}
Then $\llaves{\theta_{st}^{i}}_{i \in \N}$ is a Cauchy sequence in $\dot{H}^{1}(\RR^3)$, so $\theta_{st}^{i} \to \theta_{st}$ strongly in $\dot{H}^{1}(\RR^3)$ and $\norm{\nabla \theta_{st}}_{L^2} \leq M$.

It remains to be shown that $\theta_{st}$ is in fact a weak solution to \eqref{Estacionaria}. By Theorem \ref{th1} we have that

\begin{equation} \label{4.2.2}
    \escalar{V^{i}\cdot \grad \theta_{st}^{i+1}, \varphi}_{L^2} = \escalar{\nabla \theta_{st}^{i+1}, \nabla \varphi}_{L^2} + \escalar{f,\varphi}_{L^2} \hspace{0,5  cm} \forall \varphi \in \dot{H}^{1}(\RR^3).
\end{equation}
We will  prove the convergence of the non-linear terms, the process for the other terms  is analogous. Notice that

\begin{align*}
   \abs{ \escalar{V\cdot \grad \theta_{st}, \varphi}_{L^2} -\escalar{V^{i}\cdot \grad \theta_{st}^{i+1}, \varphi}_{L^2}} &\leq \abs{\escalar{V\cdot \grad (\theta_{st}- \theta_{st}^{i+1}), \varphi}_{L^2}}
   + \abs{\escalar{(V^{i} - V)\cdot \grad \theta_{st}^{i+1}, \varphi}_{L^2}}\\
   &= \Psi_1 + \Psi_2.
\end{align*}
Recall that $V= \nabla P - \theta_{st} \ e_3$ and $ V^{i}=\nabla P - \theta_{st}^{i} \ e_3$. By the strong convergence of $\theta_{st}^{i}$ to $\theta_{st}$ in $\dot{H}^1(\RR^3)$, we have 

\begin{align*}
    \norm{V^{i}-V}_{\dot{H}^{1}} \leq \norm{ \nabla( \theta_{st}^{i}- \theta_{st})  }_{L^{2}} \to 0, 
\end{align*}
as $i \to \infty$. Thus, $V^{i} \to V$ strongly in $\dot{H}^1(\RR^3)$ and $V \in \dot{H}^1(\RR^3)$.

By Hölder Inequality, Sobolev Embedding Theorem and Lebesgue Interpolation we have that  $ \norm{\theta_{st}}_{L^2}, \norm{ \nabla \theta_{st}}_{L^2}, \norm{\nabla \varphi}_{L^2} < \infty$. Then

\begin{equation} \label{4.2.3}
\begin{split}
   \Psi_1 \leq C\norm{\grad (\theta_{st}- \theta_{st}^{i+1})}_{L^2} \norm{\theta_{st}}_{L^2} \norm{ \nabla \theta_{st}}_{L^2} \norm{\nabla \varphi}_{L^2} \to 0,
\end{split}   
\end{equation}
because of the strong convergence of $\theta_{st}^{i}$ to $\theta_{st}$.

Analogously for $\Psi_2$, we have 

\begin{align*}
    \Psi_2 \leq \norm{V^{i} - V}_{L^{6}} \norm{\grad \theta_{st}^{i+1}}_{L^{2}} \norm{\varphi}_{L^{3}} & \leq C \norm{\nabla(V^{i} - V)}_{L^{2}} \norm{\grad \theta_{st}^{i+1}}_{L^{2}} \norm{\varphi}_{L^{3}} \\
    & \leq C \norm{\nabla(V^{i} - V)}_{L^{2}} \norm{\grad \theta_{st}^{i+1}}_{L^{2}} \norm{\nabla \varphi}_{L^{2}}^{\frac 12} \norm{\varphi}_{L^{2}}^{\frac 12}.
\end{align*}
Since  $V^{i} \to V$ strongly in $\dot{H}^1(\RR^3)$  it follows $ \Psi_2 \to 0$ as $i \to \infty$.

We can conclude that

\begin{equation*}
    \abs{ \escalar{V\cdot \grad \theta_{st}, \varphi}_{L^2} -\escalar{V^{i}\cdot \grad \theta_{st}^{i+1}, \varphi}_{L^2}} \to 0 \hspace{0,5 cm} \text{as} \hspace{0,5 cm} i \to \infty.
\end{equation*}
We have obtained the following

\begin{align*}
     \escalar{V\cdot \grad \theta_{st}, \varphi}_{L^2} - \escalar{\nabla \theta_{st}, \nabla \varphi}_{L^2}= \lim_{i \to \infty} \escalar{V^{i}\cdot \grad \theta_{st}^{i+1}, \varphi}_{L^2} - \escalar{\nabla \theta_{st}^{i+1}, \nabla \varphi}_{L^2} =  \escalar{f,\varphi}_{L^2}, 
\end{align*}
$\forall \varphi \in \dot{H}^{1}(\RR^3)$. Thus, $\theta_{st}$ is a weak solution to \eqref{Estacionaria}.

{\bf Uniqueness:} Let $\theta_{st,1}$ and $\theta_{st,2}$ two weak solutions to \eqref{Estacionaria} such that

\begin{equation*}
    \norm{\nabla \theta_{st,i}}_{L^2} \leq \norm{f}_X \hspace{0,5 cm} \text{and} \hspace{0,5 cm}\norm{(\theta_{st})_i}_{L^2} <\infty \hspace{0,5 cm} \text{for} \hspace{0,2 cm} i=1,2.
\end{equation*}
Define $Z= \theta_{st,1} - \theta_{st,2}$ and notice that it solves the following equation

\begin{align*}
    V_2\cdot\nabla Z  + (V_1 - V_2)\cdot \nabla \theta_{st,1} =\Delta Z,
\end{align*}
where $V_i = \grad P - \theta_{st,i} \ e_3$. Multiplying this equation by $Z$,  integrating by parts, and using the skew-symmetry property of the convective term, we obtain

\begin{align*}
    \norm{\nabla Z}_{L^2}^2  \leq  \frac C2 \norm{ \nabla Z}_{L^2}^2 \norm{\theta_{st,1}}_{L^3}^2 + \frac12 \norm{\nabla Z}_{L^2}^2, 
\end{align*}
where we used the Hölder Inequality, Sobolev Embedding and Young Inequality for products. By Lebesgue Interpolation for $\norm{(\theta_{st})_1}_{L^3}$, Lemma \ref{lema4.1} and 
 $\norm{\nabla \theta_{st,i}}_{L^2} \leq \norm{f}_X$ we obtain

\begin{equation*}
   \norm{\nabla Z}_{L^2}^2 \leq C\norm{ \nabla Z}_{L^2}^2 \norm{(\theta_{st})_1}_{L^3}^2 \leq  C M\norm{f}_X \norm{ \nabla Z}_{L^2}^2.
\end{equation*}
Given that $\norm{f}_X < \dfrac{1}{C M}$,  there exists  $\delta >0$ such that

\begin{equation*}
    \norm{\nabla Z}_{L^2}^2 = \norm{ \nabla Z}_{L^2}^2 + \delta,
\end{equation*}
which is a contradiction. Thus, the solution must be unique.
\end{proof}

\section{ Asymptotic Behavior of Solutions to \eqref{Parabolica}}

In this last section, we are interested in the asymptotic behavior of the solutions to \eqref{Parabolica} that are close to the stationary $\theta_{st}$. To this end, we consider equation \eqref{Parabolica} with initial data $\theta_0 = \theta_{st} + \omega_0$, that is, a perturbation to $\theta_{st}$ from \eqref{Estacionaria}. We want to establish  whether such initial datum leads to a solution $\theta$ which  tends to $\theta_{st}$ when time goes to infinity. This can be formally stated as follows: let $\omega= \theta- \theta_{st}$, where $\theta$ and $\theta_{st}$ are solutions to \eqref{Parabolica} and \eqref{Estacionaria} respectively. Then $\omega$ solves

\begin{equation}  \label{5.0.2}
            \begin{cases} 
         &\partial_{t}\omega(t)+ V_\theta\cdot\nabla \omega + V_\omega \cdot\nabla \theta_{st}  = \Delta \omega,\\
         &  V_\omega= V_\theta - V_{\theta_{st}},\\
         &\nabla \cdot V_\omega = 0,\\
         &\omega(x,0)=  \omega_0.
            \end{cases}
\end{equation}
If we prove that the solution to \eqref{5.0.2} goes to zero as $t \to \infty$ we will have shown that $\theta$ goes to $\theta_{st}$.

This section is organized as follows. We first recall, in Theorem \ref{existenciaDPM}, the existence result for solutions to \eqref{Parabolica} by Castro, Córdoba, Gancedo and Orive \cite{Castro_2009}. Regarding the existence of a weak solution to \eqref{5.0.2}, we will follow the retarded mollifiers method developed by Caffarelli, Kohn, and Nirenberg \cite{caffarelli1982partial}. We introduce a mollifier in time, $\psi_\delta$,  to regularize  the terms $V_{\theta}\cdot \nabla \omega$ and $V_{\omega}\cdot \nabla \theta_{st}$. This ensures that $\psi_{\delta}[\omega]$ and  $\psi_{\delta}[V_{\omega}]$ will depend on $\omega$ and $V_{\omega}$, respectively, only on the  past time interval $t- 2\delta <t-s < t- \delta$. Therefore, we will obtain a linear equation over each strip $I_{k}=[k\delta, (k+1)\delta]\times \RR^3$, where $\delta = \frac{\tau}{m}$ and $k=0,\dots, m-1$. By solving these equations iteratively and gluing the results, we obtain a solution $\omega_m$ to the regularized equation over $[0,\tau]\times \RR^3$. Finally, these solutions form a sequence $\llaves{\omega_{m}}_{m \in \N}$ that, via  energy inequalities, converge to a weak solution to \eqref{5.0.2}. To this end, in Lemma \ref{lm: existencia} we prove existence of solutions to these linear equations and we then, in Theorem \ref{teorema_soluciones_debiles_decay} prove convergence of $\llaves{\omega_m}_{m \in \N}$ to a solution. With respect to the decay of $\omega$, in Definition \ref{def_asymptotically} we explicitly define the concept of asymptotically stable solution. Then we demonstrate the decay of the $L^2$ norm of $\omega$ using the Fourier splitting method.
  
\subsection{Existence of weak solutions} 

In \cite{Castro_2009}, Castro, Córdoba, Gancedo, and Orive established the existence of a weak solution to \eqref{Parabolica}. More precisely, their result is stated as follows.

\begin{theorem} \label{existenciaDPM}
  Let $\theta_0 \in L^2 (\RRn)$. Then, for any $\tau >0$, there exist at least one weak solution $\theta \in C([0, \tau]; L^2(\RRn)) \cap L^2([0, \tau];H^{1}(\RRn))$ to \eqref{Parabolica}. 
\end{theorem}


Let us introduce the definition of weak solution to \eqref{5.0.2}.

\begin{definition}
We say that a function 

    \begin{equation*}
        \omega \in L^{\infty}(0, \tau; L^{2}(\RR^3)) \cap L^{2}(0, \tau; \dot{H}^{1}(\RR^3)), \hspace{0,2 cm} \partial_{t}\omega(t) \in L^2(0, \tau; \dot{H}^{-1}(\RR^3)),
    \end{equation*}
    is a weak solution to \eqref{5.0.2} if for any $\varphi \in C^{\infty}_{c}(\RR^{3} \times \RR)$

\begin{equation*}
    \begin{split}
      \int_0^{\tau} \int_{\RR^3} &\escalar{ \partial_t \omega(t), \varphi(t)}_{\dot{H}^{-1}, \dot{H}^{1}}  \ dt + \int_0^{\tau} \int_{\RR^3} (V_\theta \cdot \nabla  \omega(t))\varphi(t) \ dx \ dt\\
       &+  \int_0^{\tau}\int_{\RR^3}(V_{\omega} \cdot \nabla \theta_{st}) \varphi(t) \ dx \ dt        
        = - \int_{0}^{\tau} \escalar{ \nabla  \omega(t), \nabla \varphi(t)}_{L^2} \ dt.   
    \end{split}
\end{equation*}
\end{definition}

Following  Niche and Orive \cite{Niche-Orive} closely, we find an expression for $V_{\omega}$ in terms of $\omega$. Since $\grad \cdot V_{\omega}=0$, we have that

\begin{equation*}
- \Delta P_{\omega} = \partial_{x_3} \omega.
\end{equation*}
where $P_\omega = P_\theta - P_{\theta_{st}}$ and $P_\theta$, $P_{\theta_{st}}$ are the pressures associated to the parabolic and stationary problems respectively. Notice that this last equation correspond to a Poisson equation with external force $\partial_{x_3} \omega$ and whose solution in $\RR^3$ is  given by

\begin{equation} \label{eq 1: presion}
P_{\omega}(x) = - \frac{1}{3S_{2}}\int_{\RR^3}\frac{1}{\abs{y}}  \ \partial_{y_3} \omega(x-y) \ dy,
\end{equation}
where $S_{2}$  is the surface area of the unit ball in $\RR^3$. Integrating by parts over $\RR^3$ we get 

\begin{equation} \label{eq 2: presion}
\begin{split}
P_{\omega}(x) & = \frac{1}{3S_{2}}\text{P.V.} \int_{\RR^3}\frac{x_3 - y_3}{\abs{x-y}^3} \ \omega(y) \ dy.\\
\end{split}
\end{equation}
Applying gradient in \eqref{eq 2: presion} it follows

\begin{equation*}
\begin{split}
 \nabla P_{\omega}(x) = \frac{1}{3S_{2}} \text{P.V.} \int_{\RR^3} \mathcal{K}(x-y) \omega(y) \ dy,
\end{split} 
\end{equation*}
where

\begin{equation*}
 \mathcal{K}(x) = \parentesis{\frac{-3x_1x_3}{\abs{x}^5}, \frac{-3x_2x_3}{\abs{x}^5},\frac{\abs{x}^2 - 3x_3^2}{\abs{x}^5}}.
\end{equation*}
Thus,

\begin{equation} \label{eq 4: presion}
 V_{\omega} (x) = H_{\varepsilon}[\omega](x) = - \frac{1}{3S_{2}} \text{P.V.} \int_{\RR^3} \mathcal{K}(x-y) \omega(y) \ dy - \omega(x)e_3 .
\end{equation}
By Calderón-Zygmund Theorem (see Theorem 5.1, Chapter 5 in Duoandikoetxea \cite{javier})

\begin{equation*}
\norm{ \grad P_{\omega}}_{L^p} \leq C \norm{ \omega}_{L^p} \hspace{0,5 cm} 1 < p < \infty,
\end{equation*}
where $C$ is a constant that depends only on $p$. Therefore,

\begin{equation} \label{eq 5: presion}
\norm{V_{\omega}}_{L^p} \leq \norm{ \grad P_{\omega}}_{L^p} + \norm{\omega}_{L^{p}} \leq (C +1)\norm{\omega}_{L^{p}} \leq C \norm{\omega}_{L^{p}}, \hspace{0,5 cm} 1 < p < \infty. 
\end{equation}
This also implies that

\begin{equation} \label{eq 6: presion}
\norm{\nabla V_{\omega}}_{L^2} \leq C \norm{\nabla \omega}_{L^2}.
\end{equation}

Having established the initial properties of the $L^p$ norm of $V_{\omega}$, we now continue on the path to prove the existence of weak solutions to \eqref{5.0.2}. The following lemma establishes that there is a weak solution to the associated linear problem.

\begin{lemma} \label{lm: existencia}
Let $\tau >0$  and $v_1, v_2 \in L^{\infty}(0, \tau; L^2(\RR^3)) \cap L^2(0,\tau; \dot{H}^1(\RR^3))$ with $\nabla \cdot v_2 = 0$. Furthermore, let $\omega_0, \theta_0 \in L^2(\RR^3)$, $\theta$ a weak solution to \eqref{Parabolica} with initial datum $\theta_{0}$ given by Theorem \ref{existenciaDPM}, and $\theta_{st}$ as in Theorem \ref{existenciaestacionaria}. Then, there exists a unique 
\begin{equation*}
\omega \in L^{\infty}(0, \tau; L^{2}(\RR^3)) \cap L^{2}(0, \tau; \dot{H}^{1}(\RR^3)), \hspace{0,2 cm} \partial_{t}\omega \in L^2(0, \tau; \dot{H}^{-1}(\RR^3)),
\end{equation*}
which is a weak solution to 
 \begin{equation}   \label{eq_lineal}
            \begin{cases} 
         &\partial_{t}\omega(t)+ V_\theta\cdot\nabla v_1 + v_2 \cdot\nabla \theta_{st}  = \Delta \omega,\\        
         &\omega(x,0)=  \omega_0.
            \end{cases}
\end{equation} 
\end{lemma}

\begin{proof}

We will first prove the existence of a weak solution to \eqref{eq_lineal} on a bounded smooth domain $\Omega$ with Dirichlet boundary conditions. Then, as in Theorem \ref{th1}, we will use this to obtain a weak solution in the whole space  from a sequence of solutions defined on balls with increasing radii.

Notice that the terms $V_\theta\cdot\nabla v_1$ and $v_2 \cdot\nabla \theta_{st} $ do not depend on $\omega$, therefore, we can consider them as an external force $f$ in the equation and rewrite it as

\begin{equation} \label{eq_lineal_limitado}  
            \begin{cases} 
         &\partial_{t}\omega(t)  - \Delta \omega(t) = h \hspace{0,5 cm} \text{in}  \hspace{0,2 cm}  \Omega \times(0, \tau],\\   
         & \omega(x,t) = 0 \hspace{0,8 cm} \text{on}  \hspace{0,2 cm}  \partial\Omega,\\
         &\omega(x,0)=  \omega_0, \hspace{0,5 cm} \text{in}  \hspace{0,2 cm}  \Omega,
            \end{cases}
\end{equation}
where $h=- V_\theta\cdot\nabla v_1 - v_2 \cdot\nabla \theta_{st}$.

Let us prove that $h \in L^2(0, \tau; H^{-1}(\Omega))$. For $\varphi \in H^{1}_{0}(\Omega)$, we take

\begin{equation*}
    \escalar{h, \varphi}_{\dot{H}^{-1}, \dot{H}^{1}} = - \int_{\Omega} (V_\theta\cdot\nabla v_1 + v_2 \cdot\nabla \theta_{st}) \varphi \ dx.
\end{equation*}
Since $\nabla \cdot V_{\theta}=0$ we can use the skew-symmetry property of the convective term  to obtain

\begin{equation*}
     \escalar{h, \varphi}_{\dot{H}^{-1}, \dot{H}^{1}} = \int_{\Omega} (V_\theta\cdot\nabla \varphi) v_1 - (v_2 \cdot\nabla \theta_{st}) \varphi \ dx.
\end{equation*}
By Hölder's inequality, Sobolev Embedding, Lebesgue interpolation  and  \eqref{remark2} we have

\begin{equation*}
\begin{split}
       \norm{h(t)}_{H^{-1}(\Omega)} &\leq C \parentesis{\norm{V_{\theta}(t)}_{L^6(\Omega)}\norm{ v_1(t)}_{L^3(\Omega)}+ \norm{v_2(t)}_{L^3(\Omega)}\norm{\nabla \theta_{st}}_{L^2(\Omega)}}\\
       &\leq C \parentesis{\norm{V_{\theta}(t)}_{L^6(\RR^3)}\norm{ v_1(t)}_{L^3(\RR^3)}+ \norm{v_2(t)}_{L^3(\RR^3)}\norm{\nabla \theta_{st}}_{L^2(\RR^3)}}\\
        &\leq C\norm{\theta(t)}_{L^6(\RR^3)}\norm{v_1(t)}_{L^2(\RR^3)}^{\frac12}\norm{\nabla v_1(t)}_{L^2(\RR^3)}^{\frac12}\\        & \hspace{3 cm}+ C\norm{v_2(t)}_{L^2(\RR^3)}^{\frac12}\norm{\nabla v_2(t)}_{L^2(\RR^3)}^{\frac12}\norm{\nabla \theta_{st}}_{L^2(\RR^3)} \\  
        &\leq C \norm{\nabla \theta(t)}_{L^2(\RR^3)}\norm{v_1(t)}_{L^2(\RR^3)}^{\frac12}\norm{\nabla v_1(t)}_{L^2(\RR^3)}^{\frac12}\\        &\hspace{3 cm}+ C\norm{v_2(t)}_{L^2(\RR^3)}^{\frac12}\norm{\nabla v_2(t)}_{L^2(\RR^3)}^{\frac12}\norm{\nabla \theta_{st}}_{L^2(\RR^3)}.\\       
\end{split} 
\end{equation*}
Squaring on both sides, using Young inequality  and  integrating over $(0, \tau)$ we obtain

\begin{equation*} 
    \begin{split}
    \int_{0}^{\tau}  \norm{h(t)}_{H^{-1}(\Omega)}^2 \ dt &\leq C \int_{0}^{\tau}\norm{\nabla \theta(t)}_{L^2(\RR^3)}^{2}\norm{v_1(t)}_{L^2(\RR^3)}\norm{\nabla v_1(t)}_{L^2(\RR^3)} \ dt\\ 
    &\hspace{2,5 cm}+ C \int_{0}^{\tau}\norm{v_2(t)}_{L^2(\RR^3)}\norm{\nabla v_2(t)}_{L^2(\RR^3)}\norm{\nabla \theta_{st}}_{L^2(\RR^3)}^2\ dt.\\   
    \end{split}
\end{equation*}
Since $\theta \in L^2(0, \tau; \dot{H}^1(\RR^3))$ and $v_1, v_2 \in L^{\infty}(0, \tau; L^2(\RR^3)) \cap L^2(0,\tau; \dot{H}^1(\RR^3))$ it follows that

\begin{equation*}
    \norm{v_i(t)}_{L^2(\RR^3)}\norm{\nabla v_i(t)}_{L^2(\RR^3)} < C, \hspace{0,5 cm} \text{for a.e $t\in (0, \tau]$, with $i=1,2$}   
\end{equation*}
and

\begin{equation*}
    \int_{0}^{\tau} \norm{\nabla \theta(t)}_{L^2(\RR^3)}^{2} \ dt < \infty.
\end{equation*}
Since Theorem \ref{existenciaestacionaria} implies that $\|\nabla \theta_{st}\|_{L^2} < M$, it follows that

\begin{equation} \label{eq_lineal_limitada_2}
    \begin{split}
    \int_{0}^{\tau}  \norm{h(t)}_{H^{-1}(\Omega)}^2 \ dt \leq C \int_{0}^{\tau}\norm{\nabla \theta(t)}_{L^2(\RR^3)}^{2} \ dt 
   + C M^2 \tau.  
    \end{split}
\end{equation}
Therefore $h \in L^2(0,\tau; H^{-1}(\Omega))$.

By the arguments in Theorem 9.6 in Section 9.3 and in Section 9.4 in Salsa and Verzini \cite{salsa_libro}, there is a unique weak solution $\omega$ to \eqref{eq_lineal_limitado}, and the following energy estimates hold, for every $t \in [0, \tau]$

\begin{equation} \label{des_energia_lineal_1}
        \norm{\omega(t)}_{L^2(\Omega)}^{2} + \int_{0}^{t} \norm{\nabla \omega(s)}_{L^2(\Omega)}^2 ds\leq \norm{\omega_0}_{L^2(\Omega)}^{2} + \int_{0}^{t} \norm{h(s)}_{H^{-1}(\Omega)} ds,
\end{equation}
and
\begin{equation} \label{des_energia_lineal_2}
        \ \int_{0}^{t} \norm{\partial_s \omega(s)}_{H^{-1}(\Omega)}^2 ds\leq 2\norm{\omega_0}_{L^2(\Omega)}^{2} + 4\int_{0}^{t} \norm{h(s)}_{H^{-1}(\Omega)} ds.
\end{equation}
We proved existence of a solution in a smooth bounded domain $\Omega$. To extend the result to the whole space we proceed as in Theorem \ref{th1}, by taking a sequence  of balls with radii $R_n$ going to infinity and using a diagonal argument to show that a global solution exists. We omit the details. 

 Finally by the lower  semi-continuity of the norm, together with \eqref{des_energia_lineal_1} and \eqref{des_energia_lineal_2}  it follows that

 \begin{equation*}         \norm{\omega(t)}_{L^2(\RR^3)}^{2} + \int_{0}^{t} \norm{\nabla \omega(s)}_{L^2(\RR^3)}^2 ds\leq \norm{\omega_0}_{L^2(\RR^3)}^{2} + \int_{0}^{t} \norm{h(s)}_{H^{-1}(\Omega)} ds
\end{equation*}
and

\begin{equation*} 
        \ \int_{0}^{t} \norm{\partial_s \omega(s)}_{\dot{H}^{-1}(\RR^3)}^2 ds\leq 2\norm{\omega_0}_{L^2(\RR^3)}^{2} + 4\int_{0}^{t} \norm{h(s)}_{H^{-1}(\Omega)} ds.
\end{equation*}
\end{proof}

Having established the existence of a weak solution to the associated linear problem, we now introduce the mollifier $\psi_\delta$ that we will use to turn \eqref{5.0.2} into a suitable form to apply Lemma \ref{lm: existencia}. This is important because the regularity  of solutions obtained in Theorem \ref{teorema_soluciones_debiles_decay} will allow us to apply the Fourier Splitting Method rigorously to prove decay rates. We choose $\psi \in C^{\infty} (\RR)$ such that

\begin{equation*}
\psi \geq 0, \hspace{0,5 cm} \text{supp $ \psi$} \subset [1,2], \hspace{0,5 cm} \text{and} \hspace{0,5 cm} \int_{ \RR} \psi(y,s)ds=1.
\end{equation*}
For a given function $g=g(x,t)$ with $x \in \RR^3$ and $t \in [0, \tau]$, we define

\begin{equation*}
\tilde{g}(x,t) = 
\begin{cases}
g(x,t), \hspace{1 cm} &x \in \RR^3, t \in [0, \tau],\\
0, \hspace{1 cm} &x \in \RR^3, t \in (\tau, \infty).
\end{cases}
\end{equation*} 
Now, we set

\begin{equation} \label{eq: 1 Mollifier}
\psi_{\delta}[g](x,t)= \frac{1}{\delta}\int_{\RR} \psi\parentesis{\frac{s}{\delta}} \tilde{g}(x,t- s) \ ds, \hspace{0,5 cm} t>0.
\end{equation}
Using the change of variable $s'= \dfrac{s}{\delta}$ in \eqref{eq: 1 Mollifier} we obtain

\begin{equation*} 
\psi_{\delta}[g](x,t)= \int_{\RR} \psi\parentesis{s'} \tilde{g}(x,t- \delta s') \ dy \ ds'.
\end{equation*}
Since  $ \text{supp $ \psi$}\subset [1,2]$ , it follows that if $s' \in (0,1)\cup (2, \infty)$,  then $\psi_{\delta}(g)(x,t)= 0$. On the other hand, when $t- 2\delta < t-s' < t - \delta$, given that $\tilde{g}$  vanishes when $t-s < 0$, it follows that $\psi_{\delta}(g)(x,t)= 0$ for $\psi \neq 0$. Therefore, the values of $\psi_{\delta}[g](x,t)$ are determined by the past times $t-s$ of $\tilde{g}(x,s)$, that satisfy $t- 2\delta < t-s < t - \delta$.

The following Lemma, from Caffarelli, Kohn and Nirenberg \cite{caffarelli1982partial}, provides some properties of $\psi_{\delta}$.

\begin{lemma} \label{lm: existencia 2}

Let $v \in L^{\infty}(0, \tau; L^2(\RR^3))\cap  L^{2}(0, \tau; \dot{H}^1(\RR^3)) $ such that $\nabla \cdot v = 0$, then

\begin{equation*}
    \nabla \cdot \psi_{\delta}[v]= 0,
\end{equation*}
\begin{equation} \label{eq_mollifier_1}
   \sup_{0 < t < \tau} \norm{\psi_{\delta}[v](t)}_{L^2}  \leq C \sup_{0 < t < \tau} \norm{v(t)}_{L^2},
\end{equation}
\begin{equation} \label{eq_mollifier_2}
    \int_{0}^{\tau} \norm{ \nabla \psi_{\delta}[v](t)}_{L^2}^{2} \ dt \leq C  \int_{0}^{\tau} \norm{ \nabla v(t)}_{L^2}^{2} \ dt.
\end{equation}
Furthermore, $\psi_{\delta}[v] \in L^{\infty}(0, \tau; L^2(\RR^3))\cap  L^{2}(0, \tau; \dot{H}^1(\RR^3))$.

\end{lemma}

Using \eqref{eq 4: presion} we now define

\begin{equation*}
\psi_{\delta}[V_{\omega_m}]= \psi_{\delta}[H_{\varepsilon}(\omega_m)],
\end{equation*}
which leads to the linear version of equation \eqref{5.0.2} in an interval $I_k=\RR^3\times [ k\delta, (k+1) \delta]$

\begin{equation} \label{eq 7: existencia}
\partial_{t} \omega_m + V_{\theta} \cdot \nabla \psi_{\delta}[\omega_m] + \psi_{\delta}[V_{\omega_{m}}] \cdot \nabla \theta_{st} = \Delta \omega_m, \hspace{0,5 cm} \omega_m(x,0)= \omega_0 (x). 
\end{equation}
By Lemma \ref{lm: existencia 2},  
\begin{displaymath}
    \psi_{\delta}[V_{\omega}], \psi_{\delta}[\omega] \in  L^{\infty}(0, \tau; L^2(\RR^3))\cap  L^{2}(0, \tau; \dot{H}^1(\RR^3))
    \end{displaymath}
provided that  $\omega \in L^{\infty}(0, \tau; L^2(\RR^3))\cap L^{2}(0, \tau; \dot{H}^1(\RR^3))$. Therefore, we can use  Lemma \ref{lm: existencia} to obtain a weak solution $\omega_{m,0}=\omega_{m}\big\rvert_{I_0}$ to
 
\begin{equation*} 
\begin{cases}
&\partial_{t} \omega_{m,0} + V_{\theta} \cdot \nabla \psi_{\delta}[\omega_{m,0}] + \psi_{\delta}[V_{\omega_{m,0}}] \cdot \nabla \theta_{st} = \Delta \omega_{m,0}, \hspace{0,5 cm} \text{in} \hspace{0,5 cm}  I_0 =\RR^3\times  [0,  \delta],\\
& \omega_{m,0}(x,0)= \omega_0 (x). 
\end{cases}
\end{equation*}
Applying  Lemma \ref{lm: existencia} again, but now with $\omega_{m,0}$ as initial data, we obtain a weak solution $\omega_{m,1}=\omega_{m}\big\rvert_{I_1}$ to
 
\begin{equation*} 
\begin{cases}
&\partial_{t} \omega_{m,1} + V_{\theta} \cdot \nabla \psi_{\delta}[\omega_{m,1}] + \psi_{\delta}[V_{\omega_{m,1}}] \cdot \nabla \theta_{st} = \Delta \omega_{m,1}, \hspace{0,5 cm} \text{in} \hspace{0,5 cm}  I_1=\RR^3\times [\delta, 2 \delta],\\
& \omega_{m,1}(x,0)= \omega_{m,0} (x, \delta). 
\end{cases}
\end{equation*}
Proceeding inductively, we obtain a unique weak solution to \eqref{eq 7: existencia} given by
\begin{equation} \label{eq 8: existencia}
\begin{split}
       &\omega_m \in L^{\infty}(0, \tau; L^{2}(\RR^3)) \cap L^{2}(0, \tau; \dot{H}^{1}(\RR^3)),\\
       &\omega_{m}(x,t) = \omega_{m,k}(x,t) \hspace{0,2 cm} \text{if} \hspace{0,2 cm} (x,t) \in \RR^3\times [ k\delta, (k+1) \delta],
\end{split}
\end{equation}
over $\RR^3 \times [0, \tau]$, with $\delta = \frac{\tau}{m}$,  $k = 0, \dots, m-1$, and for $\tau >0$.

In the next Lemma we establish the existence of a weak solution to \eqref{5.0.2}, based on the previously described construction.

\begin{theorem} \label{teorema_soluciones_debiles_decay}
Let $\theta_0, \omega_0 \in L^2(\RR^3)$, and let $f$ and $\theta_{st}$ be as in Theorem \ref{existenciaestacionaria}.  Then there is a weak solution to \eqref{5.0.2}.
\end{theorem}

\begin{proof}
    Let $\llaves{\omega_m}_{m \in \N}$ be the sequence of solutions defined in \eqref{eq 8: existencia}. Multiplying \eqref{eq 7: existencia} by $\omega_m$, integrating by parts over space, and integrating in time we obtain

\begin{equation} \label{eq 12: existencia}
\begin{split}
    \frac12 \norm{\omega_m(t)}_{L^2} + \int_{0}^{t} \norm{\nabla \omega_m (s)}_{L^2}^{2} ds &= \norm{\omega_0}_{L^2} - \int_{0}^{t} \int_{\RR^3}(\psi_{\delta}[V_{\omega_{m}}] \cdot \nabla \theta_{st}) \omega_m \ dx\ ds\\
    & \hspace{1,5 cm}- \int_{0}^{t} \int_{\RR^3}(V_{\theta}\cdot \nabla \psi_{\delta}[\omega_{m}]) \omega_m \ dx\ ds.
\end{split}
\end{equation}
By Hölder Inequality, Sobolev Embedding  and  Lebesgue Interpolation we have

\begin{equation} \label{eq 13: existencia}
\begin{split}
 - \int_{0}^{t} \int_{\RR^3}(\psi_{\delta}[V_{\omega_{m}}] \cdot \nabla \theta_{st}) \omega_m \ dx\ ds & \leq 
  \int_{0}^{t} \norm{\psi_{\delta}[V_{\omega_{m}}]}_{L^6}  \norm{\theta_{st}}_{L^3}  \norm{\nabla \omega_m}_{L^2}\\
 &\leq \int_{0}^{t} C \norm{\nabla \psi_{\delta}[V_{\omega_{m}}]}_{L^2}  \norm{\theta_{st}}_{L^3}  \norm{\nabla \omega_m}_{L^2} ds\\ 
 &\leq \int_{0}^{t} C \norm{\nabla \psi_{\delta}[V_{\omega_{m}}]}_{L^2} \norm{\theta_{st}}_{L^2}^{\frac 12} \norm{\nabla \theta_{st}}_{L^2}^{\frac 12}  \norm{\nabla \omega_m}_{L^2} ds\\ 
 &\leq \int_{0}^{t} C M \norm{f}_{\dot{H}^{-1}} \norm{\nabla \psi_{\delta}[V_{\omega_{m}}]}_{L^2}  \norm{\nabla \omega_m}_{L^2} ds,\\ 
\end{split}
\end{equation}
where we have used Theorem \ref{existenciaestacionaria}. Analogously, by  \eqref{remark2} we have that $         \norm{V_{\theta}}_{L^p} \leq C  \norm{\theta(t)}_{L^p}$ for $1< p < \infty$, then

\begin{equation*} 
\begin{split}
 - \int_{0}^{t} \int_{\RR^3}(V_{\theta} \cdot \nabla\psi_{\delta}[\omega_{m}]) \omega_m &\ dx\ ds \leq  \int_{0}^{t} \norm{V_{\theta}}_{L^3}  \norm{\nabla\psi_{\delta}[\omega_{m}]}_{L^2}\norm{\omega_m(s)}_{L^6} \ ds\\
   &\leq  \int_{0}^{t} C \norm{ \theta(s)}_{L^3} \norm{\nabla \psi_{\delta}[\omega_{m}]}_{L^2} \norm{\nabla \omega_m(s)}_{L^2} \ ds\\
   & \leq \int_{0}^{t} C \norm{ \theta(s)}_{L^2}^{\frac 12} \norm{\nabla \theta(s)}_{L^2}^{\frac 12}  \norm{\nabla \psi_{\delta}[\omega_{m}]}_{L^2}\norm{\nabla \omega_m(s)}_{L^2} \ ds.
\end{split}
\end{equation*}
Given that $\theta \in C([0, \tau]; L^2(\RR^3)) \cap L^2([0, \tau]; H^{1}(\RR^3))$ for $\tau >0$, there exists $t_0 \in (0, t)$ that depends on $\theta_0$ such that 

\begin{equation} \label{estimativa_theta}
  \norm{\theta(s)}_{L^3}\leq C \norm{\nabla \theta(t)}_{L^2}^{\frac 12} \norm{\theta(t)}_{L^2}^{\frac 12} < C \hspace{0,5 cm} \forall \ s \in (t_0 , t),
\end{equation}
which implies that

\begin{equation} \label{eq 14: existencia}
     - \int_{0}^{t} \int_{\RR^3}(V_{\theta} \cdot \nabla\psi_{\delta}[\omega_{m}]) \omega_m \ dx\ ds \leq C \int_{0}^{t} \norm{\nabla \psi_{\delta}[\omega_{m}]}_{L^2}\norm{\nabla \omega_m(s)}_{L^2} \ ds.
\end{equation}
From  \eqref{eq 13: existencia}, \eqref{eq 14: existencia} and by Young inequality for products it follows that we estimate the right hand of \eqref{eq 12: existencia} through

\begin{equation} \label{eq 14.1: existencia}
    \begin{split}
          &-  \int_{0}^{t}  \int_{\RR^3}(\psi_{\delta}[V_{\omega_{m}}] \cdot \nabla \theta_{st}) \omega_m \ dx\ ds - \int_{0}^{t} \int_{\RR^3}(V_{\theta} \cdot \nabla\psi_{\delta}[\omega_{m}]) \omega_m \ dx\ ds\\
         &\leq C\int_{0}^{t}  M \norm{f}_{\dot{H}^{-1}} \norm{\nabla \psi_{\delta}[V_{\omega_{m}}]}_{L^2}  \norm{\nabla \omega_m(s)}_{L^2} ds + C\int_{0}^{t} \norm{\nabla \psi_{\delta}[\omega_{m}]}_{L^2}\norm{\nabla \omega_m(s)}_{L^2} \ ds\\
         & = C\int_{0}^{t} \parentesis{ M \norm{f}_{\dot{H}^{-1}} \norm{\nabla \psi_{\delta}[V_{\omega_{m}}]}_{L^2} +   \norm{\nabla \psi_{\delta}[\omega_{m}]}_{L^2}}  \norm{\nabla \omega_m(s)}_{L^2} \ ds\\
         & \leq \frac C2 \int_{0}^{t} \parentesis{M \norm{f}_{\dot{H}^{-1}} \norm{\nabla \psi_{\delta}[V_{\omega_{m}}]}_{L^2} + C  \norm{\nabla \psi_{\delta}[\omega_{m}]}_{L^2}}^2 ds + \frac12 \int_{0}^{t} \norm{\nabla \omega_m(s)}_{L^2}^{2} \ ds\\
           & \leq  C M^2 \norm{f}_{\dot{H}^{-1}}^2 \int_{0}^{t}  \norm{\nabla \psi_{\delta}[V_{\omega_{m}}]}_{L^2}^2 \ ds + C \int_{0}^{t}  \norm{\nabla \psi_{\delta}[\omega_{m}]}_{L^2}^2\ ds + \frac12 \int_{0}^{t} \norm{\nabla \omega_m(s)}_{L^2}^{2} \ ds.
    \end{split}
\end{equation}
Plugging \eqref{eq 14.1: existencia} into \eqref{eq 12: existencia}  we obtain

\begin{equation}  \label{eq 14.2: existencia}
\begin{split}
     \frac12 \norm{\omega_m(t)}_{L^2} + \frac12 \int_{0}^{t} \norm{\nabla \omega_m (s)}_{L^2}^{2} ds
 \leq \norm{\omega_0}_{L^2}^2 &+ C M^2 \norm{f}_{\dot{H}^{-1}}^2 \int_{0}^{t}  \norm{\nabla \psi_{\delta}[V_{\omega_{m}}]}_{L^2}^2 ds\\
 & \hspace{1 cm}+ C\int_{0}^{t} \norm{\nabla \psi_{\delta}[\omega_{m}]}_{L^2}^2 \ ds.\\
\end{split}
\end{equation}
By Lemma \ref{lm: existencia 2} and \eqref{eq 6: presion} we have that

\begin{equation*}
    \int_{0}^{t} \norm{\nabla \psi_{\delta}[V_{\omega_{m}}](s)}_{L^2}^2 \ ds \leq C_1 \int_{0}^{t} \norm{\nabla \omega_m (s)}_{L^2}^2 ds,
\end{equation*}
and

\begin{equation*}
    \int_{0}^{t}  \norm{\nabla \psi_{\delta}[\omega_{m}](s)}_{L^2}^2 \ ds \leq C_2  \int_{0}^{t} \norm{\nabla \omega_m(s)}_{L^2}^2 ds.
\end{equation*}
Therefore, 

\begin{equation*} 
      \norm{\omega_m(t)}_{L^2} +  \int_{0}^{t} \norm{\nabla \omega_m (s)}_{L^2}^{2} ds
 \leq 2\norm{\omega_0}_{L^2}^2 + 2(C M^2 \norm{f}_{\dot{H}^{-1}}^2 + C)\int_{0}^{t} \norm{\nabla \omega_m(s)}_{L^2}^2 ds.
\end{equation*}
Let us choose $\norm{f}_{\dot{H}^{-1}}^2$ small enough such that $1- 2(C_1 M^2 \norm{f}_{\dot{H}^{-1}}^2 + C_2) > 0 $. This implies that 

 \begin{equation*}  
\begin{split}
      \norm{\omega_m(t)}_{L^2} +  \int_{0}^{t} \norm{\nabla \omega_m (s)}_{L^2}^{2} ds
 &\leq \frac{2}{C}\norm{\omega_0}_{L^2}^2, 
\end{split}
\end{equation*}
where $C = 1- (2C M^2 \norm{f}_{\dot{H}^{-1}}^2 + C)$. It follows that

 \begin{equation} \label{eq 15: existencia} 
      \sup_{0 \leq t \leq  \tau}  \norm{\omega_m(t)}_{L^2} +  \int_{0}^{\tau} \norm{\nabla \omega_m (s)}_{L^2}^{2} ds
 \leq \frac{2}{C}\norm{\omega_0}_{L^2}^2, 
 \end{equation}
 for all $m \in \N$ and any $\tau >0$. Therefore,   $\llaves{\omega_m}_{m \in \N}$ is uniformly bounded in $L^{\infty}((0,\tau);L^2(\RR^3))\cap L^2((0,\tau); \dot{H}^{1}(\RR^3))$, for any $\tau >0$.

From the weak formulation of \eqref{eq 7: existencia}, the skew-symmetry property of the convective term, and given that $\nabla \cdot V_{\theta} = 0$ and $\nabla \cdot \psi_{\delta}[V_{\omega_{m}}]=0$, it follows that

\begin{equation} \label{eq 16: existencia}
\begin{split}
    \abs{\escalar{\partial_{t} \omega_m(t), \varphi}_{\dot{H}^{-1}, {\dot{H}^{1}}}} \leq \int_{\RR^3}\abs{(V_{\theta} \cdot \nabla \varphi) \psi_{\delta}[\omega_m](t)} \ dx + \int_{\RR^3}\abs{(\psi_{\delta}[V_{\omega_{m}}(t)] \cdot \nabla \varphi) \theta_{st}} \  dx\\ + \abs{\escalar{ \nabla \omega_m(t), \nabla \varphi}_{L^2}}, 
\end{split}   
\end{equation}
for $\varphi \in \dot{H}^{1}(\RR^3)$. By Hölder Inequality, Sobolev Embedding, Lebesgue Interpolation,  and \eqref{estimativa_theta} it follows that

\begin{equation}\label{eq 17: existencia}
\begin{split}
\int_{\RR^3}&\abs{(V_{\theta} \cdot \nabla \varphi) \psi_{\delta}[\omega_m(t)]} \ dx \leq \norm{V_{\theta}}_{L^3} \norm{\nabla \varphi}_{L^2} \norm{\psi_{\delta}d[\omega_m(t)]}_{L^6}\\
&\leq C \norm{V_{\theta}}_{L^3} \norm{\nabla \varphi}_{L^2} \norm{\nabla\psi_{\delta}[\omega_m(t)]}_{L^2}
\leq C \norm{\theta(t)}_{L^3} \norm{\nabla \varphi}_{L^2} \norm{\nabla \psi_{\delta}[\omega_m(t)]}_{L^2}\\
&\leq C \norm{\nabla \varphi}_{L^2} \norm{\nabla \psi_{\delta}[\omega_m](s)}_{L^2},
\end{split}
\end{equation}
for all $ s\in (t_0, t)$. Analogously,

\begin{equation*}
\int_{\RR^3}\abs{(\psi_{\delta}[V_{\omega_{m}}] \cdot \nabla \varphi) \theta_{st}} \  dx \leq C \norm{\nabla \psi_{\delta}[V_{\omega_{m}}]}_{L^2} \norm{\theta_{st}}_{L^3} \norm{\nabla \varphi}_{L^2}.
\end{equation*}
By the Lebesgue Interpolation Theorem  and Theorem \ref{existenciaestacionaria} it follows that

\begin{equation} \label{eq 18: existencia}
\int_{\RR^3}\abs{(\psi_{\delta}[V_{\omega_{m}}] \cdot \nabla \varphi) \theta_{st}} \  dx \leq C M \norm{f}_{\dot{H}^{-1}} \norm{\nabla \psi_{\delta}[V_{\omega_{m}}]}_{L^2}\norm{\nabla \varphi}_{L^2}.
\end{equation}
From Cauchy-Schwarz Inequality we also have

\begin{equation}\label{eq 19: existencia}
\abs{\escalar{ \nabla \omega_m, \nabla \varphi}_{L^2}} \leq \norm{\nabla \omega_m(t)}_{L^2} \norm{\nabla \varphi}_{L^2}.
\end{equation}
Substituting \eqref{eq 17: existencia}, \eqref{eq 18: existencia} and \eqref{eq 19: existencia} into \eqref{eq 16: existencia} yields 

\begin{equation} \label{eq 21: existencia} 
   \norm{\partial_t\omega_m(t)}_{\dot{H}^{-1}} \leq C \norm{\nabla \psi_{\delta}[\omega_m](t)}_{L^2} + CM \norm{f}_{\dot{H}^{-1}} \norm{\nabla \psi_{\delta}[V_{\omega_{m}}](t)}_{L^2}+ \norm{\nabla \omega_m(t)}_{L^2}, 
\end{equation}
and then

\begin{equation} \label{eq 22: existencia} 
\begin{split}
     \int_{0}^{\tau} \norm{\partial_s\omega_m(s)}_{\dot{H}^{-1}}^2 ds &\leq  \frac{C}{2} \int_{0}^{\tau}\norm{\nabla \psi_{\delta}[\omega_m](s)}_{L^2}^{2} \ ds\\
     &   +  C M^2 \norm{f}_{\dot{H}^{-1}}^2\int_{0}^{\tau}  \norm{\nabla \psi_{\delta}[V_{\omega_{m}}](s)}_{L^2}^2 ds
     + \int_{0}^{\tau} \norm{\nabla \omega_m(s)}_{L^2}^{2} \ ds. 
\end{split}
\end{equation}
Substituting \eqref{eq 15: existencia} into \eqref{eq 22: existencia} and applying Lemma \ref{lm: existencia 2}, we obtain

\begin{equation}\label{des_energia_derivada}
 \begin{split}
     \int_{0}^{\tau} \norm{\partial_s\omega_m(s)}_{\dot{H}^{-1}}^2 ds \leq   C M^2 \norm{f}_{\dot{H}^{-1}}^2 \int_{0}^{\tau}\norm{\nabla \omega_m(s)}_{L^2}^{2} \ ds 
      +  \frac{2}{C}\norm{\omega_0}_{L^2}^2
      \leq C\norm{\omega_0}_{L^2}^2,
\end{split}
\end{equation}
for any $\tau >0$. Inequality \eqref{des_energia_derivada} implies that $\llaves{\partial_s\omega_m}$ is uniformly bounded in $L^2([0,\tau]; \dot{H}^{-1}(\RR^3))$. It follows that there is a subsequence of $\llaves{\omega_m}_{m \in \N}$ such that

\begin{align*}
    \omega_m \rightharpoonup \omega \ \text{weakly in }  \ L^{2}(0, \tau; \dot{H}^1(\RR^3)), \hspace{1 cm} \partial_s \omega_m \rightharpoonup \partial_t \omega\ \text{weakly in } \ L^2(0,\tau; \dot{H}^{-1}(\RR^3)), 
\end{align*}
and
\begin{equation*}
    \omega_m \rightharpoonup \omega \ \text{weakly-$\ast$ in }  \ L^{\infty}(0,\tau; \dot{H}^1(\RR^3)),
\end{equation*}
for all $\tau > 0$. Notice that for any $r >0$ the sequence $\llaves{\omega_{m}\mid_{B_r}}_{m \in \N}$, with $B_r= B(0,r)$, satisfies the hypothesis of  the Aubin-Lions Lemma (see Theorem 4.3, Chapter IV in  Robinson, Rodrigo, and Sadowski \cite{Robinson_Rodrigo_Sadowski_2016}), then 

\begin{equation*}
    \omega_m \to \omega \hspace{0,5 cm} \text{strongly in}  \ L^{2}(0, \tau; L^2(B_r)), \hspace{0,5 cm} \forall r>0.
\end{equation*}
So, 

\begin{equation*}
    \omega_m \to \omega \hspace{0,5 cm} \text{strongly in }  \ L^{2}(0,  \tau; L_{loc}^2(\RR^3)).
\end{equation*}
Particularly, given $\varphi \in C^{\infty}_c (\RR^3)$ we have 

\begin{equation*}
    \omega_m \to \omega \hspace{0,5 cm} \text{strongly in }  \ L^{2}(0,  \tau; L^2(\hat{\Omega})),
\end{equation*}
where $\hat{\Omega}= \text{supp}(\varphi)$. Since strong convergence implies weak convergence, we have that

\begin{equation*}
    \omega_m  \rightharpoonup \omega \ \text{weakly in }  \ L^{2}(0,  \tau; L^2(\hat{\Omega})),
\end{equation*}
for all $\tau > 0$.

We now pass to  the limit in

\begin{equation*}
    \begin{split}
       \int_0^{\tau} \escalar{ \partial_t \omega_m(t), \varphi(t)}_{\dot{H}^{-1}, \dot{H}^{1}} &\ dt +  \int_0^{\tau}\int_{\RR^3} (V_\theta \cdot \nabla  \psi_{\delta}[\omega_m](t))\varphi(t) \ dx \ dt\\
       &+ \int_0^{\tau} \int_{\RR^3}(\psi_{\delta}[V_{\omega_m}](t) \cdot \nabla \theta_{st}) \varphi(t)dx \ dt
       =  - \int_{0}^{\tau} \escalar{ \nabla  \omega_m(t), \nabla \varphi(t)}_{L^2} \ dt, \\        
    \end{split}
\end{equation*}
for $\varphi \in C^{\infty}_{c} ([0, \tau)\times \RR^3)$,  to prove that the limit function $\omega$ is in fact a weak solution to \eqref{5.0.2}. We focus on the terms containing the mollified functions, as the convergence of the other terms follows from standard arguments. Let $\varphi \in C^{\infty}_{c} ([0, \tau)\times \RR^3)$, by Cauchy-Schwarz inequality we have

\begin{equation} \label{eq 28: existencia}
    \begin{split}        \int_{0}^{\tau}\int_{\RR^3}\abs{((\psi_{\delta}[V_{\omega_m}](t)-V_{\omega}(t)) \cdot \nabla \theta_{st}) \varphi(t)}dx \ d t\leq C_{\varphi}  \norm{\nabla \theta_{st}}_{L^2}\int_{0}^{\tau}\norm{(\psi_{\delta}[V_{\omega_m}]-V_{\omega})(t)}_{L^2}  \ dt,
    \end{split}
\end{equation}
where $ \norm{\nabla \varphi}_{L^{\infty}(0, \tau; L^{\infty}(\RR^3))} < C_{\varphi}$. Notice that

\begin{equation*}
\begin{split}
    \psi_{\delta}[V_{\omega_m}](t)- V_{\omega}(t) &= \int_{\RR} \psi(s) H_{\varepsilon}[\omega_m](x, t- \delta s) \ ds -  H_{\varepsilon}[\omega](x, t)\\
    &= \int_{\RR} \psi(s) H_{\varepsilon}[\omega_m](x, t- \delta s)  \ ds -  H_{\varepsilon}[\omega](x, t)\int_{\RR} \psi(s)\ ds\\
    &= \int_{\RR} \psi(s) \parentesis{H_{\varepsilon}[\omega_m](x, t- \delta s) -  H_{\varepsilon}[\omega](x, t)} \ ds.
   \end{split}
\end{equation*}
Then, by the Minkowski inequality for integrals and the linearity of the operator $H_{\varepsilon}$ follows that

\begin{equation*}
    \norm{ \psi_{\delta}[V_{\omega_m}](t)- V_{\omega}(t)}_{L^2} \leq \int_{\RR} \psi(s) \norm{ H_{\varepsilon}[\omega_m(x, t- \delta s)-  \omega(x, t)]}_{L^2} \ ds.
\end{equation*}
By \eqref{eq 5: presion} we get

\begin{equation} \label{eq 28.1: existencia}
\norm{\psi_{\delta}[V_{\omega_m}](t)-  V_{\omega}(t)}_{L^2} \leq C \int_{\RR} \psi(s) \norm{\omega_m(t- \delta s) - \omega(t)}_{L^2} ds.
\end{equation}
Plugging \eqref{eq 28.1: existencia} into  \eqref{eq 28: existencia}, Fubini's Theorem  implies that

\begin{align*}
\int_{0}^{\tau}\int_{\RR^3}\abs{((\psi_{\delta}[V_{\omega_m}]-V_{\omega}) \cdot \nabla \theta_{st}) \varphi(t)} dx \ &ds \leq  C_{\varphi} \norm{\nabla \theta_{st}}_{L^2} \int_{0}^{\tau} \int_{\RR}\psi(s) \norm{\omega_m(t- \delta s) - \omega(t)}_{L^2} ds \ dt\\
&\leq  C_{\varphi} \norm{\nabla \theta_{st}}_{L^2} \int_{0}^{\tau} \int_{\RR} \psi(s) \norm{\omega_m(t- \delta s) - \omega(t)}_{L^2} ds \ dt\\
&\leq  C_{\varphi} \norm{\nabla \theta_{st}}_{L^2} \int_{\RR} \psi(s) \int_{0}^{\tau}  \norm{\omega_m(t- \delta s) - \omega(t)}_{L^2} dt \ ds,\\
\end{align*}
where $C_{\varphi}$ is a constant depending on $\varphi$ and on the constant given by  \eqref{eq 5: presion}. Then 

\begin{equation} \label{des_triangular}
\begin{split}
\int_{0}^{\tau}\int_{\RR^3}&\abs{((\psi_{\delta}[V_{\omega_m}](t)-V_{\omega}(t)) \cdot \nabla \theta_{st}) \varphi(t)}dx \ ds\\
&\leq C  \norm{\nabla \theta_{st}}_{L^2} \int_{\RR} \psi(s) \int_{0}^{\tau}  \norm{\omega_m(t- \delta s) - \omega(t- \delta s)}_{L^2} + \norm{\omega (t- \delta s) - \omega(t)}_{L^2} dt \ ds\\
& =  C  \norm{\nabla \theta_{st}}_{L^2} \int_{\RR} \psi(s) \int_{0}^{\tau}  \norm{\omega_m(t- \delta s) - \omega(t- \delta s)}_{L^2} dt \ ds\\
& \hspace{5 cm}+ C  \norm{\nabla \theta_{st}}_{L^2}\int_{\RR} \psi(s) \int_{0}^{\tau} \norm{\omega (t- \delta s) - \omega(t)}_{L^2} dt \ ds.
\end{split}
\end{equation}
Recalling that $\delta= \frac{\tau}{m}$, let us define the following sequences

\begin{equation*}
\begin{split}
     \Gamma_{m,1} &=  \psi(s)\int_{0}^{\tau}  \norm{\omega_m(t- \delta s) - \omega(t- \delta s)}_{L^2} dt,\\
     \Gamma_{m,2} &=  \psi(s) \int_{0}^{\tau} \norm{\omega \parentesis{t- \frac{\tau}{m} s} - \omega(t)}_{L^2} dt.
\end{split}   
\end{equation*}
Since $ \omega_m \to \omega$ strongly in   $L^{2}(0,  \tau; L^{2}_{loc}(\RR^3))$, the Hölder Inequality implies that

\begin{equation*}
    \int_{0}^{\tau}  \norm{\omega_m(t- \delta s) - \omega(t- \delta s)}_{L^2(\hat{\Omega})} dt \leq \tau \psi(s)\parentesis{\int_{0}^{\tau}  \norm{\omega_m(t- \delta s) - \omega(t- \delta s)}_{L^2(\hat{\Omega})}^{2} dt}^{\frac12}.
\end{equation*}
where $\hat{\Omega} = \text{supp} \ \varphi$. Thus $\Gamma_{m,1} \to 0$ as $m \to \infty$ pointwise. Given that  $\Gamma_{m,1}$ is bounded in $L^{\infty}(0, \tau; L^2(\RR^3))$,  it follows that

\begin{equation} \label{des_triangular_est_1}
   \lim_{m\to \infty}\int_{\RR} \psi(s)\int_{0}^{\tau}  \norm{\omega_m(t- \delta s) - \omega(t- \delta s)}_{L^2(\hat{\Omega})} dt = 0.
\end{equation}
On the other hand, since  the norm and the translations are continuous in $L^2(0, \tau; L^{2}_{loc}(\RR^3))$, the limit $\dfrac{\tau}{m} \to 0$, as $m \to \infty$ implies that 

\begin{equation*}
    \lim_{m\to \infty} \psi(s) \parentesis{\int_{0}^{\tau} \norm{\omega \parentesis{t- \frac{\tau}{m} s} - \omega(t)}_{L^2(\hat{\omega})}^{2} dt}^{\frac12} = 0.
\end{equation*}
By Hölder Inequality it follows that

\begin{equation*}
    \int_{0}^{\tau}  \norm{\omega \parentesis{t- \frac{\tau}{m} s} - \omega(t)}_{L^2} dt \leq \tau \psi(s)\parentesis{\int_{0}^{\tau}  \norm{\omega \parentesis{t- \frac{\tau}{m} s} - \omega(t)}_{L^2(\hat{\Omega})}^{2} dt}^{\frac12}.
\end{equation*}
Thus, $\Gamma_{m,2} \to 0$ as $m \to \infty$ pointwise.

Given that  $\Gamma_{m,2}$ is bounded in $L^{\infty}(0, \tau; L^2(\RR^3))$, by the  Dominated Convergence Theorem we have that

\begin{equation} \label{des_triangular_est_2}
   \lim_{m\to \infty}\int_{\RR} \psi(s)\int_{0}^{\tau}  \norm{\omega \parentesis{t- \frac{\tau}{m} s} - \omega(t)}_{L^2(\hat{\Omega})} dt = 0.
\end{equation}
By \eqref{des_triangular_est_1} and \eqref{des_triangular_est_2}, passing to the limit in \eqref{des_triangular} implies that

\begin{equation*}
    \lim_{m\to\infty} \int_{0}^{\tau}\int_{\RR^3}(\psi_{\delta}[V_{\omega_m}](t) \cdot \nabla \theta_{st}) \varphi(t)dx \ ds = \int_{0}^{\tau}\int_{\RR^3}(\psi_{\delta}[V_{\omega}] \cdot \nabla \theta_{st}) \varphi(t)dx \ ds.
\end{equation*}
Proceeding analogously, we have that

\begin{equation*}
    \lim_{m\to\infty} \int_{0}^{\tau}\int_{\RR^3}(V_{\theta}\cdot \psi_{\delta}[\omega_m]) \varphi(t)dx \ ds = \int_{0}^{\tau}\int_{\RR^3}(V_{\theta}\cdot \omega(t)) \varphi(t)dx \ ds.
\end{equation*}
Therefore,

\begin{equation*}
    \begin{split}
      \int_0^{\tau} &\escalar{ \partial_t \omega(t), \varphi(t)}_{\dot{H}^{-1}, \dot{H}^{1}} dt +  \int_0^{\tau} \int_{\RR^3} (V_\theta \cdot \nabla  \omega(t))\varphi(t) dx \ dt
       +  \int_0^{\tau} \int_{\RR^3}(V_{\omega} \cdot \nabla \theta_{st}) \varphi(t)dx \ dt\\
       &=\lim_{m\to\infty} \int_0^{\tau} \escalar{ \partial_t \omega_m(t), \varphi(t)}_{\dot{H}^{-1}, \dot{H}^{1}} \ dt + \int_0^{\tau}\int_{\RR^3} (V_\theta \cdot \nabla  \psi_{\delta}[\omega_m](t))\varphi(t) dx \ dt\\
       & \hspace{7 cm}+ \int_0^{\tau}\int_{\RR^3}(\psi_{\delta}[V_{\omega_m}] \cdot \nabla \theta_{st}) \varphi(t)dx \ dt\\
       & =  - \lim_{m\to\infty}\int_{0}^{\tau} \escalar{ \nabla  \omega_m(t), \nabla \varphi(t)}_{L^2} \ dt
       = - \int_{0}^{\tau} \escalar{ \nabla  \omega(t), \nabla \varphi(t)}_{L^2} \ dt,  
    \end{split}
\end{equation*}
for all $\varphi \in C^{\infty}_{c}(\RR^3 \times \RR)$. We conclude that there exists a weak solution $\omega$ to \eqref{5.0.2}.

By the weak lower semi-continuity of the norm in Banach spaces, together with the weak convergence of the sequence $\llaves{\omega_m}_{m \in \N}$, it follows from \eqref{eq 15: existencia} and \eqref{des_energia_derivada}, that

\begin{equation*}
     \sup_{0 \leq t \leq  \tau}  \norm{\omega(t)}_{L^2} +  \int_{0}^{\tau} \norm{\nabla \omega (s)}_{L^2}^{2} ds,
 \leq C\norm{\omega_0}_{L^2}^2
\end{equation*}
and

\begin{equation*}
      \int_{0}^{\tau} \norm{\partial_s\omega(s)}_{\dot{H}^{-1}}^2 ds \leq C \norm{\omega_0}_{L^2}^2.
\end{equation*}
Consequently, 

 \begin{equation*}
\omega \in L^{\infty}(0, \tau; L^{2}(\RR^3)) \cap L^{2}(0, \tau; \dot{H}^{1}(\RR^3)), \hspace{0,2 cm} \partial_t\omega \in L^2(0, \tau; \dot{H}^{-1}(\RR^3)),
\end{equation*}
for $\tau > 0$.
\end{proof}

\subsection{Decay properties and Asymptotic behavior}

In this section, we analyze the asymptotic behavior of the weak solution $\omega$ to \eqref{5.0.2}. Our approach employs the  Fourier Splitting Method developed by  Schonbek in \cite{ME1980, ME1985}. 

\begin{definition} \label{def_asymptotically}

A solution $\theta_{st}\in \dot{H}^{1}(\RR^3)$ to \eqref{Estacionaria} with $f \in \dot{H}^{-1}(\RR^3)$, as established in Theorem \ref{existenciaestacionaria}, is called  {\it asymptotically stable}, if 
     \begin{equation*}
         \lim_{t \to \infty}  \norm{\theta(t)- \theta_{st}}_{L^2} =0,
     \end{equation*} 
for any solution $\theta (t)$ to \eqref{Parabolica}.      
\end{definition}

We will prove that, in fact, the $L^2$ norm of the solution  $\omega$  to \eqref{5.0.2} is bounded by $(1+t)^{- \alpha}$, where $\alpha$ depends on the decay characters of $\theta_0$ and $\theta_{st}$. From this will follows that $\theta_{st}$ is asymptotically stable. This is the content of Theorem \ref{Teorema_Decay_Asintotica}, which we now prove.

\begin{proof} [Proof of Theorem \ref{Teorema_Decay_Asintotica}]

Theorem \ref{teorema_soluciones_debiles_decay} implies that 

 \begin{equation} \label{eq regularidad}
\omega \in L^{\infty}(0, \tau; L^{2}(\RR^3)) \cap L^{2}(0, \tau; \dot{H}^{1}(\RR^3)), \hspace{0,2 cm} \partial_t\omega \in L^2(0, \tau; \dot{H}^{-1}(\RR^3)).
\end{equation}
Thus,  $\partial_t \omega(t), \nabla \omega(t) \in \dot{H}^{-1}(\RR^3)$ for a.e. $t \in [0, \tau]$ and, since $\dot{H}^{-1}(\RR^3) \subset \mathcal{D'}(\RR^3)$, we can interpret these linear functionals as distributions. In particular, we observe that

\begin{equation*}
   \escalar{\Delta \omega(t), \varphi}_{\dot{H}^{-1}, \dot{H}^{1}}= \escalar{\nabla \cdot \nabla \omega(t), \varphi}_{\dot{H}^{-1}, \dot{H}^{1}} = -\escalar{ \nabla \omega(t), \nabla \varphi}_{\dot{H}^{-1}, \dot{H}^{1}}.  
\end{equation*}
It follows that $\Delta \omega(t) \in \dot{H}^{-1}(\RR^3)$ for a.e. $t \in[0, \tau]$. Moreover, since $\dot{H}^{1}(\RR^3)$ is a Hilbert space, by the Riesz-Fréchet Theorem we can make the following identification

\begin{equation*}
    -\escalar{ \nabla \omega(t), \nabla \varphi}_{\dot{H}^{-1}, \dot{H}^{1}} = -\escalar{ \nabla \omega(t), \nabla \varphi}_{L^2},
\end{equation*}
for a.e $t \in [0, \tau]$ and for all $\varphi \in \dot{H}^{1}(\RR^3)$. 
 
As 

 \begin{equation*}
     V_\theta(t) \cdot\nabla \omega + V_{\omega}(t) \cdot \nabla \theta_{st} = \Delta \omega(t)- \partial_{t} \omega(t)
 \end{equation*}
 we have that  $ V_\theta(t) \cdot\nabla \omega + V_{\omega}(t) \cdot \nabla \theta_{st} \in \dot{H}^{-1}(\RR^3)$ for a.e $t \in [0, \tau]$, where

 \begin{equation*}
     \escalar{ V_\theta(t) \cdot\nabla \omega + V_{\omega}(t) \cdot \nabla \theta_{st}, \varphi}_{\dot{H}^{-1}, \dot{H}^{1}} = \int_{\RR^3} (V_\theta(t) \cdot\nabla \omega(t)) \varphi \ dx +\int_{\RR^3}  (V_{\omega}(t) \cdot \nabla \theta_{st}) \varphi \ dx. 
 \end{equation*}
 Now, as $\omega \in L^{2}(0,\tau; \dot{H}^{1}(\RR^3))$, we have

 \begin{equation} \label{formulaicion_funcional} 
 \begin{split}
   \escalar{\partial_t \omega(t), \omega(t)}_{\dot{H}^{-1}, \dot{H}^{1}} = -\escalar{ \nabla \omega(t), \nabla \omega(t)}_{L^2} &- \int_{\RR^3} (V_{\theta}(t)\cdot \nabla \omega(t) ) \omega(t) \  dx\\
   & \hspace{2 cm}- \int_{\RR^3}  (V_{\omega}(t)\cdot \nabla \theta_{st}) \omega(t) \ dx,  
 \end{split}
\end{equation}
for a.e $t \in( 0, \tau]$ and any $\tau >0$.

Let us recall that the spaces $L^2(\RR^3)$, $\dot{H}^{1}(\RR^3)$ and $\dot{H}^{-1}(\RR^3)$ form a Hilbert triplet. From \eqref{eq regularidad} and Theorem 100 in section 7.12 of Salsa and Verzini \cite{salsa_libro} we obtain

\begin{equation} \label{integracion_partes_eq}
    \int_{s}^{t} \escalar{\partial_r\omega(r), v(r)}_{\dot{H}^{-1}, \dot{H}^{1}} + \escalar{\partial_r v(r), \omega(r)}_{\dot{H}^{-1}, \dot{H}^{1}}  \ dr = \escalar{\omega(t),v(t)}_{L^2} - \escalar{\omega(s),v(s)}_{L^2},
\end{equation}
for any $v \in L^{2}(0, \tau; \dot{H}^{1}(\RR^3)), \hspace{0,2 cm} \partial_tv\in L^2(0, \tau; \dot{H}^{-1}(\RR^3))$.
Notice that, by  Hölder Inequality it follows 

\begin{equation*}
\begin{split}
    \int_{0}^{\tau} &\abs{\escalar{\partial_r\omega(r), v(r)}_{\dot{H}^{-1}, \dot{H}^{1}} + \escalar{\partial_r v(r), \omega(r)}_{\dot{H}^{-1}, \dot{H}^{1}}}  \ dr\\
&\leq \int_{0}^{\tau} \norm{\partial_r\omega(r)}_{\dot{H}^{-1}} \norm{v(r)}_{\dot{H}^{1}} + \norm{\partial_rv(r)}_{\dot{H}^{-1}} \norm{\omega(r)}_{\dot{H}^{1}} \ dr \\
& \leq \parentesis{\int_{0}^{\tau} \ \norm{\partial_r\omega(r)}_{\dot{H}^{-1}}^{2}dr}^{\frac 12}\parentesis{\int_{0}^{\tau} \ \norm{v(r)}_{\dot{H}^{1}}^{2}dr}^{\frac 12}\\
& \hspace{0,5 cm}+ \parentesis{\int_{0}^{\tau} \ \norm{\partial_rv(r)}_{\dot{H}^{-1}}^{2}dr}^{\frac 12}\parentesis{\int_{0}^{\tau} \ \norm{\omega(r)}_{\dot{H}^{1}}^{2}dr}^{\frac 12}.\\
\end{split}
\end{equation*}
Then, $\escalar{\partial_r\omega(r), v(r)}_{\dot{H}^{-1}, \dot{H}^{1}} + \escalar{\partial_r v(r), \omega(r)}_{\dot{H}^{-1}, \dot{H}^{1}} \in L^1(0, \tau)$. Therefore, we can apply the Lebesgue Differentiation Theorem to \eqref{integracion_partes_eq} to obtain 

\begin{equation} \label{integracion_partes_eq_1}
 \frac{d}{dt} \escalar{\omega(t), v(t)}_{L^2} = \escalar{\partial_t\omega(t), v(t)}_{\dot{H}^{-1}, \dot{H}^{1}} + \escalar{\partial_t v(t), \omega(t)}_{\dot{H}^{-1}, \dot{H}^{1}},    
\end{equation}
for a.e $t \in [0, \tau]$ and any $\tau >0$. Taking $v(t)= \omega(t)$ leads to

\begin{equation*}
    \frac12 \frac{d}{dt} \norm{\omega(t)}_{L^2}^{2} = \escalar{\partial_t\omega(t), \omega(t)}_{\dot{H}^{-1}, \dot{H}^{1}},
\end{equation*}
for a.e $t \in [0, \tau]$ and any $\tau >0$. Consequently, we can write \eqref{formulaicion_funcional} as

\begin{equation} \label{FS1}
\frac12 \frac{d}{dt} \norm{\omega(t)}_{L^2}^2 = -  \norm{\nabla \omega(t)}_{L^2}^2 - \int_{\RR^3}  (V_{\theta}(t)\cdot \nabla \omega(t)) \, \omega(t) dx - \int_{\RR^3}  (V_{\omega}(t)\cdot \nabla \theta_{st}) \, \omega(t) dx.
\end{equation}
By Hölder Inequality, Sobolev Embedding Theorem, and \eqref{remark2} we have

\begin{equation} \label{FS2}
- \int_{\RR^3} (V_{\theta}(t)\cdot \nabla \omega(t)) \,  \omega(t) dx \leq C \norm{\theta(t)}_{L^3} \norm{\nabla \omega(t)}_{L^2}^2.
\end{equation}
To estimate the third term in the right hand side, we use  the skew-symmetry property of the convective term,  to obtain

\begin{equation}
\label{FS3}
\begin{split}
- \int_{\RR^3}  (V_{\omega}(t)\cdot \nabla \theta_{st}) \, \omega(t) \, dx &= \int_{\RR^3}  (V_{\omega}(t)\cdot \nabla \omega(t)) \, \theta_{st} \, dx \leq C \norm{\theta_{st}}_{L^{3}} \norm{\omega(t)}_{L^6} \norm{ \nabla \omega(t)}_{L^2}\\
&\leq C \norm{\nabla \theta_{st}}_{L^{2}}^{\frac12} \norm{\theta_{st}}_{L^2}^{\frac12} \norm{\nabla \omega(t)}_{L^2}^2 
\leq CM^{\frac12} \norm{f}_{\dot{H}^{-1}}^{\frac12} \norm{\nabla \omega(t)}_{L^2}^2,
\end{split}
\end{equation}
where we used Hölder Inequality, \eqref{eq 5: presion}, Sobolev Embedding Theorem, Lebesgue Interpolation and Theorem \ref{existenciaestacionaria}.
Replacing \eqref{FS2} and \eqref{FS3} into \eqref{FS1} we have that

\begin{equation} \label{FS4}
\frac12 \frac{d}{dt} \norm{\omega(t)}_{L^2}^2 \leq - ( 1 - C \norm{\theta(t)}_{L^3} - C  M^{\frac12} \norm{f}_{\dot{H}^{-1}}^{\frac12}) \norm{\nabla \omega}_{L^2}^2.
\end{equation}
As $\theta \in C([0, \tau]; L^2(\RR^3)) \cap L^2([0, \tau]; \dot{H}^{1}(\RR^3))$, there is $t_0 \in (0, \tau)$, that depends on $\theta_0$, such that 

\begin{equation*}
   \norm{\theta(t)}_{L^3}\leq C\norm{\nabla \theta(t)}_{L^2}^{\frac 12} \norm{\theta(t)}_{L^2}^{\frac 12} < \infty \hspace{0,5 cm} \forall \hspace{0,2 cm} t \in (t_0 , \tau).
\end{equation*}
Then we can choose $\norm{f}_{\dot{H}^{-1}}^{\frac12}$ small enough such that

\begin{equation*} 
A = 1 - C \norm{\theta(t)}_{L^3} - C  M^{\frac12} \norm{f}_{\dot{H}^{-1}}^{\frac12} >0.
\end{equation*}
We now apply the Fourier Splitting Method. Let $B(t)$ a ball centered at the origin and  with a time-dependent radius $R(t)$. By  Plancherel Theorem we have that 

\begin{align*}
- A &\norm{\fourier{\nabla \omega(t)}}_{L^2}^2 = -A \int_{B(t)\cup B(t)^c} \abs{\xi}^2 \abs{ \fourier{\omega}(\xi,t)}^2 \ d\xi\\
& = - A \int_{B(t)}  \abs{\xi}^2 \abs{ \fourier{\omega}(\xi,t)}^2 \ d\xi - A \int_{ B(t)^c} \abs{\xi}^2 \abs{ \fourier{\omega}(\xi,t)}^2 \ d\xi
\leq - A \int_{ B(t)^c} \abs{\xi}^2 \abs{ \fourier{\omega}(\xi,t)}^2 \ d\xi\\
&\leq - A R^2 \int_{ B(t)^c} \abs{ \fourier{\omega}(\xi,t)}^2 \ d\xi\
 \leq - A R^2 \int_{ \RR^3} \abs{ \fourier{\omega}(\xi,t)}^2 \ d\xi + A R^2 \int_{ B(t)} \abs{ \fourier{\omega}(\xi,t)}^2 \ d\xi.
\end{align*}
So

\begin{equation*}
 \frac12 \frac{d}{dt} \norm{\omega(t)}_{L^2}^{2} + A R^2 \norm{\omega(t)}_{L^2}^2 \leq A R^2 \int_{B(t)} \abs{ \fourier{\omega}(\xi,t)}^2 \ d\xi.
\end{equation*}
We now define the following time-dependent radius

\begin{equation*}
R(t) = \parentesis{\frac{1}{A(1+t)}}^{\frac12},
\end{equation*}
and obtain

\begin{equation} \label{FS5}
\frac12 \frac{d}{dt} \norm{\omega(t)}_{L^2}^{2} +  \frac{1}{(1+t)} \norm{\omega(t)}_{L^2}^2 \leq  \frac{1}{(1+t)}  \int_{B(t)} \abs{ \fourier{\omega}(\xi,t)}^2 \ d\xi.
\end{equation}
From

\begin{equation*}
\frac{d}{dt}\fourier{\omega}(\xi,t)+ i \xi \cdot \fourier{V_\theta(t)\omega(t)} + i \xi \cdot \fourier{V_\omega(t)  \theta_{st}}  = \abs{\xi}^2\fourier{\omega(t)},
\end{equation*}
We have

\begin{equation}\label{FS6}
\fourier{\omega}(\xi,t)= e^{-\abs{\xi}^2t}\fourier{\omega_0}(\xi) + \int_{0}^t i \xi \cdot \fourier{V_\theta(s)\omega(s)} + i \xi \cdot \fourier{V_\omega(s)  \theta_{st}} \ ds,
\end{equation}
and then

\begin{equation} \label{FS7}
\fourier{\omega}(\xi,t)^2 \leq 2 \parentesis{e^{-\abs{\xi}^2t}\fourier{\omega_0}(\xi)}^2 + 2\parentesis{\int_{0}^t i \xi \cdot \fourier{V_\theta(s)\omega(s)} + i \xi \cdot \fourier{V_\omega(s)  \theta_{st}} \ ds}^2.
\end{equation}
Integrating \eqref{FS7}  over $B(t)$ yields

\begin{equation}\label{FS8}
\begin{split}
   \int_{B(t)}\abs{\fourier{\omega}(\xi,t)}^2 \ d\xi &\leq 2 \int_{B(t)} \abs{e^{-\abs{\xi}^2t}\fourier{\omega_0}(\xi)}^2 \ d\xi\\
   & \ \ \ \ + 2 \int_{B(t)}\abs{\int_{0}^t i \xi \cdot \fourier{V_\theta(s)\omega(s)} + i \xi \cdot \fourier{V_\omega(s)  \theta_{st}}\ ds}^2 d\xi. 
\end{split}
\end{equation}
To estimate the second term on the right hand side of \eqref{FS8}, we use Cauchy-Schwarz inequality and \eqref{remark2} to obtain

\begin{align} \label{FS9}
\abs{i \xi \cdot \fourier{V_\theta(s)\omega(s)}}  & \leq \abs{\xi} \abs{\fourier{V_\theta(s)\omega(s)}} \leq \abs{\xi}\norm{\fourier{V_\theta(s)\omega(s)}}_{L^{\infty}} \leq \abs{\xi}\norm{V_\theta(s)\omega(s)}_{L^{1}} \notag \\
& \leq \abs{\xi} \norm{V_{\theta}(s)}_{L^2}\norm{\omega(s)}_{L^2} \leq C \abs{\xi} \norm{\theta(s)}_{L^2}\norm{\omega(s)}_{L^2}\leq C \abs{\xi} \norm{\omega(s)}_{L^2}.
\end{align}
Analogously, by Theorem \ref{existenciaestacionaria} it follows that

\begin{equation}\label{FS10}
\abs{i\xi \cdot \fourier{V_\omega(s)  \theta_{st}}}  \leq C \abs{\xi} \norm{\omega(s)}_{L^2}\norm{\theta_{st}}_{L^2} \leq  C M \abs{\xi}\norm{\omega(s)}_{L^2}.
\end{equation}
By \eqref{FS9} and \eqref{FS10}, we obtain

\begin{equation} \label{FS11}
\abs{ i \xi \cdot \fourier{V_\theta(s)\omega(s)} + i \xi \cdot \fourier{V_\omega(s)  \theta_{st}}}^2  \leq C \abs{\xi}^2 \norm{\omega(s)}_{L^2}^2.
\end{equation}
Applying Jensen's inequality together with \eqref{FS11} to the second term on the right hand side of \eqref{FS8} we obtain

\begin{equation*}
\begin{split}
\abs{\int_{0}^{t} i \xi \cdot \widehat{V_\theta(s)\omega(s)} + i \xi \cdot \widehat{V_\omega(s) \theta_{st}}   \ ds}^2  &\leq t \int_0^t \abs{ i \xi \cdot \widehat{V_\theta(s)\omega(s)} + i \xi \cdot \widehat{V_\omega(s)\theta_{st}}}^2 \ ds \\ & 
\leq C  \abs{\xi}^2 t \int_{0}^t \norm{\omega(s)}_{L^2}^2 \ ds.
\end{split}
\end{equation*}
This yields

\begin{equation} \label{FS12}
 \int_{B(t)}\abs{\int_{0}^t i \xi \cdot \fourier{V_\theta(s)\omega(s)} + i \xi \cdot \fourier{V_\omega(s)  \theta_{st}} \ ds}^2 d\xi\leq C\parentesis{\int_{B(t)}\abs{\xi}^2 \ d\xi}  t \int_{0}^t \norm{\omega(s)}_{L^2}^2 \ ds.
\end{equation}
From  Theorem \ref{th2.5} it follows that

\begin{equation} \label{FS13}
    \int_{B(t)} e^{- 2 \abs{\xi}^2t}\abs{\fourier{\omega_0}(\xi)}^2 d\xi \leq \norm{e^{- \abs{\xi}^2t}\fourier{\omega_0}(\xi)}_{L^2}^2 \leq C (1+t)^{-\frac{3}{2} - r^{*}}.
\end{equation}
Using \eqref{FS12} and \eqref{FS13} in \eqref{FS8} we get

\begin{equation}\label{FS14}
\int_{B(t)}\abs{\fourier{\omega}(\xi,t)}^2 \ d\xi \leq C (1+t)^{-\frac 32 - r^{*}} +   C\parentesis{\int_{B(t)}\abs{\xi}^2 \ d\xi}  t \int_{0}^t \norm{\omega(s)}_{L^2}^2 \ ds.
\end{equation}
Now we choose $\beta > \max \llaves{ \frac{5}{2} + r^{*}, \frac{5}{2}}$  and use the integrating factor $h(t)=(1+t)^{\beta}$ to turn \eqref{FS5} into

\begin{equation*}
\frac{d}{dt}\parentesis{\norm{\omega(t)}_{L^2}^{2}(1+t)^{\beta}} \leq C(1+t)^{\beta-1}\parentesis{(1+t)^{-\frac32-r^{*}}+(1+t)^{-\frac32}\int_{0}^{t}\norm{\omega(s)}_{L^2}^{2} \ ds}.    
\end{equation*}
From this we obtain

\begin{equation} \label{FS21}
\norm{\omega(t)}_{L^2}^{2} (1+t)^{\frac 32} \leq  C (1+t)^{- r^{*}} +  C  \int_{0}^t  (1+s)^{-\frac 32} \norm{\omega(s)}_{L^2}^2 (1+s)^{\frac 32} \  ds.
\end{equation}
We can re-write \eqref{FS21} as

\begin{equation*}
X(t) \leq a(t) + \int_0^t b(s) X(s) \ ds,
\end{equation*}
where

\begin{equation*}
\begin{split}
&X(t) = \norm{\omega(t)}_{L^2}^{2} (1+t)^{\frac 32}, \hspace{0,5 cm} a(t)= C (1+t)^{- r^{*}}, \hspace{0,5 cm}
 b(t) = C (1+t)^{-\frac 32}.
\end{split}
\end{equation*}
We consider two cases:  first, we focus on the case  where $r^{*} \leq 0$, which implies that $a(t)$ is a non-decreasing function of $t$. Given that  $a(t), b(t), X(t)$  are continuous on $[0, t]$, it follows from Corollary 1.2. on page 4 of Bainov and Simeonov \cite{bainov1992} that

\begin{equation*}
X(t) \leq a(t) \exp{\int_0^{t}b(s)ds}.
\end{equation*}
Notice that

\begin{equation*}
\int_0^t b(s) \ ds= \int_0^t C (1+s)^{-\frac 32} \ ds  < \infty,
\end{equation*}
 Therefore,

\begin{equation*} 
\norm{\omega(t)}_{L^2}^{2}  \leq C (1+t)^{- \frac 32 - r^{*}}. 
\end{equation*}
Now suppose $r^{*} > 0$. From \eqref{FS21} we obtain 

\begin{equation} \label{FS22}
\norm{\omega(t)}_{L^2}^{2}\leq  C (1+t)^{- \parentesis{ \frac32 + r^{*}}} +  C (1+t)^{-\frac 32}  \int_{0}^t   \norm{\omega(s)}_{L^2}^2 \  ds.
\end{equation}
Then let

\begin{equation*}
\begin{split}
X(t) = \norm{\omega(t)}_{L^2}^{2}, \hspace{0,2 cm} a(t)= C (1+t)^{- \parentesis{\frac32 + r^{*}}}, \hspace{0,2 cm}
 b(t) = (1+t)^{-\frac 32}, \hspace{0,2 cm} k(t)= 1.
\end{split}
\end{equation*}
From Theorem 1, page 356 in Mitrinovi\'c, Pe\v{c}ari\'c, and Fink \cite{fink2013} we obtain

\begin{equation*}
    X(t) \le a(t) + b(t)\int_{0}^{t}a(s)k(s)\exp\parentesis{\int_{s}^{t}b(r)k(r) \ dr} \ ds,
\end{equation*}
which leads to
\begin{equation*}
    \begin{split}
        \norm{\omega(t)}_{L^2}^{2} \leq C(1+t)^{-\parentesis{\frac32 + r^{*}}} + C(1+t)^{- \frac32}\int_{0}^{t} (1+s)^{-\parentesis{\frac 32 + r^{*}}} \ ds \leq C(1+t)^{- \frac 32}. 
    \end{split}
\end{equation*}
Then
\begin{equation*}
    \norm{\omega(t)}_{L^2}^{2} \leq C(1+t)^{- \min\llaves{\frac32, \frac32 + r^{*}}}.
\end{equation*}
As, from Lemma \ref{suma_decay_character} we have that  $r^{*}(\omega_0) \leq \min\llaves{r^{*}(\theta_0), r^{*}(\theta_{st})}$, we proved our result.

\end{proof}

\bibliographystyle{plain}
\bibliography{bibliografia}

\end{document}